\documentclass[a4paper,11pt,reqno]{amsart}

\usepackage{amsmath,amssymb,amsfonts,amscd,amsthm,mathtools,float,multicol}
\usepackage[pdftex]{graphicx}
\usepackage{caption}
\usepackage[usenames]{color}
\usepackage{ifpdf}

\newcommand{\Llra}{\Longleftrightarrow}

\newcommand{\Lra}{\Longrightarrow}

\newcommand{\lra}{\longrightarrow}
\newcommand{\ra}{\rightarrow}

\newcommand{\E}{\exists}

\newcommand{\Pj}{\mb{P}}
\newcommand{\bd}{\mathbf}
\newcommand{\mb}{\mathbb}
\newcommand{\ol}{\overline}
\newcommand{\0}{\emptyset}
\newcommand{\Cl}{\operatorname{Cl}}
\newcommand{\dv}{\operatorname{div}}
\newcommand{\suml}{\sum\limits}
\newcommand{\M}{\ol{M}}
\newcommand{\Pic}{\operatorname{Pic}}
\newcommand{\prodl}{\prod\limits}
\newcommand{\B}{\textbf}
\newcommand{\Eff}{\operatorname{\overline{Eff}}}

\newcommand{\col}{\,|\,}
\newcommand{\Z}{\mb{Z}}
\newcommand{\R}{\mb{R}}

\newcommand{\bp}{\chi^{-1}}
\newcommand{\ap}{\phi_{*}^{-1}}
\newcommand{\I}{^{-1}}
\newcommand{\Aff}{\mb{A}}
\newcommand{\sm}{\smallsetminus}
\newcommand{\balpha}{\bar{\phi}}
\newcommand{\bbeta}{\bar{\chi}}

\newcommand{\z}{{\bf a}}
\newcommand{\zb}{ {\bf b } }
\newcommand{\Lbo}{\Lambda_{{\bf b}}^0}
\newcommand{\Lao}{\Lambda_{\z}^0}
\newcommand{\vphi}{\varphi}
\newcommand{\Diag}{\operatorname{Diag}}
\newcommand{\NS}{\operatorname{NS}}

\newtheoremstyle{mystyle}{}{}{\itshape}{}{\scshape}{.}{ }{}
\theoremstyle{mystyle}

\swapnumbers
\newtheorem{Theorem}{Theorem}[section]

\newtheorem{Proposition}[Theorem]{Proposition}
\newtheorem{Lemma}[Theorem]{Lemma}

\newtheorem{Corollary}[Theorem]{Corollary}
\newtheorem{Claim}[Theorem]{Claim}
\newtheorem{Conjecture}[Theorem]{Conjecture}

\newtheoremstyle{myreview}{}{}{}{}{\scshape}{.}{ }{}
\theoremstyle{myreview}

\newtheorem{Example}[Theorem]{Example}

\newtheorem{Remark}[Theorem]{Remark}

\newcounter{et}[Theorem]

\newcommand{\C}{\operatorname{C}}
\begin{document}

\title[Extremal divisors on moduli spaces of rational curves]{Extremal divisors on moduli spaces of rational curves with marked points}
\author{Morgan Opie}\thanks{ This research was supported by NSF grant DMS-1001344 (PI Jenia Tevelev). }
\begin{abstract}

We study effective divisors on $\M_{0,n}$, focusing on hypertree divisors introduced by Castravet and Tevelev, and the proper transforms of divisors on $\M_{1,n-2}$ introduced by Chen and Coskun. We relate these two types of divisors and exhibit divisors on $\ol{M}_{0,n}$ for $n \geq 7$ that furnish counterexamples to a conjectural description of the effective cone of $\M_{0,n}$ given by Castravet and Tevelev.\end{abstract}
\maketitle

%-------------------------------------BACKGROUND------------------------------------------------

\tableofcontents
\section{Introduction}\label{intro}
 The moduli space $M_{0,n}$ parameterizes equivalence classes of $n$ distinct marked points on $\mb{P}^1$ under the action of $PGL_2$. We will be primarily concerned with $\M_{0,n}$, the Deligne--Mumford compactification of $M_{0,n}$ by stable rational curves with $n$ marked points. The Deligne-Mumford compactification parameterizes nodal trees of $\mb{P}^1$'s with $n$ markings such that each component has at least $3$ ``special" points (markings or nodes), modulo automorphisms. 

\begin{figure}[h]\label{rationalcurves}
\center
\includegraphics[scale=.32]{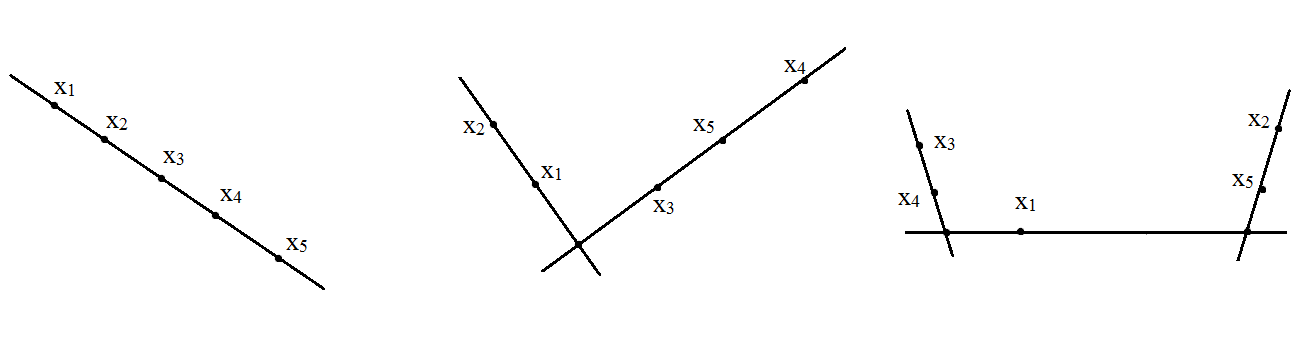}
\caption{Examples of stable rational curves, n=5.}
\end{figure} The locus $\M_{0,n} \sm M_{0,n}$ is a union of boundary divisors, defined as follows: for $I \subset \{1,\dots, n\}$ with both $I$ and $\{1,\dots,n\} \sm I$ of size at least two, the boundary divisor $\delta_I$ consists of classes of stable rational curves in $\ol{M}_{0,n} \sm M_{0,n}$ with a node separating the markings corresponding to indices in $I$ and $\{1,\dots,n\} \sm I$. \par

Significantly, $\M_{0,n}$ can be realized as an iterated blow-up of $\Pj^{n-3}$ via a Kapranov morphism. Any Kapranov morphism restricts to an ismorphism of $M_{0,n}$ with its image, and any boundary divisor is contracted by some Kapranov morphism.  Hence each boundary divisor generates an extremal ray of the effective cone of $\M_{0,n}$, and select boundary divisors together with the pull-back of a hyperplane class under a Kapranov morphism comprise free generators for the class group $\Cl(\M_{0,n})$ \cite{K}. We will use these Kapranov generators throughout the paper.  

\par In \ref{basics}, we describe a method of specifying divisors on $\M_{0,n}$ via polynomials in $n$ variables. We discuss how to compute the classes of these divisors, and include Macaulay2 code to compute classes. While useful for checking results on $\M_{0,n}$ with $n\leq 10$, the code is not practical for large $n$. 

In \ref{ht}, we recall the definitions of hypertrees and hypertree divisors from \cite{CT}. A major result of \cite{CT} is that hypertree divisors corresponding to ``irreducible" hypertrees are exceptional divisors of some birational contraction, and hence generate extremal rays of the effective cone of $\M_{0,n}$. In \cite{CT}, it is further speculated that 

\begin{Conjecture}\label{conj} The effective cone of $\ol{M}_{0,n}$ is generated by boundary divisors 
and by divisors parameterized by irreducible hypertrees and the pull-backs of these divisors under forgetful morphisms. 
\end{Conjecture} 
This motivates us to study hypertree divisors and their classes. We generalize a result of \cite{CT} to obtain polynomials specifying all hypertree divisors, and use our Macaulay2 program to compute all irreducible hypertree divisor classes on $\M_{0,n}$ for $6 \leq n \leq 10$. We then turn our attention to other effective divisors.

In \cite{CC}, Chen and Coskun construct divisors on $\M_{1,n}$ using $n$-tuples $(a_1,\dots,a_n)$ of integers such that $\sum a_i =0$. They show that if $\gcd(a_1,\dots,a_n)=1$, the divisor corresponding to the $n$-tuple is a rigid, extremal effective divisor. We examine the proper transforms of these divisors on $\M_{0,n+2}$ with respect to the clutching morphism that glues the two markings (we call these proper transforms Chen--Coskun divisors). We first find formulas for the classes of Chen--Coskun divisors, and then prove results relating Chen and Coskun and hypertree divisors. In particular, we show that the Chen--Coskun divisor associated to the $n$-tuple $(1,1,\dots,-1,-1,\dots)$ coincides with a particular hypertree divisor.  \par

Next, we investigate extremality of Chen--Coskun divisors. Such divisors need not be extremal, as examples in \ref{rex} show. However, in \ref{counterex}, we show that the Chen--Coskun divisor corresponding to $(n,1,-1,-1,-1,\dots)$ is always non-boundary extremal. Moreover, these particular Chen--Coskun divisors are neither a hypertree divisors nor pull-backs of hypertree divisors. Hence, they furnish a counterexamples to the conjecture \ref{conj}. \par

In \ref{conditions} we give a proof of a well-known criterion for extremality used in \ref{counterex}. In fact, we show that our criterion not only guarantees extremality, but also that the effective cone is ``not rounded" near the given divisor. No reference for this fact was found. \par

 In \ref{rex}, we further investigate extremality and rigidity of Chen and Coskun classes. We give criteria for rigidity and non-extremality via conditions on the $n$-tuple defining a Chen--Coskun divisor, and discuss implications for constructing ``large" families of extremal divisors on $\M_{0,n}$.  \\

\noindent\textsc{Acknowledgements.} I am grateful to Jenia Tevelev: this paper began as a summer 2013 REU project under his instruction at the University of Massachusetts Amherst, and his guidance was instrumental in its creation. I also want to thank Anna Kazanova, Tassos Vogiannou, and Julie Rana for helping me to learn Macaulay2 and debug programs; Ana-Maria Castravet for discussion during the 2013 Young Mathematicians Conference at the Ohio State University; Stephen Coughlan and Eduardo Cattani for  feedback on earlier drafts; and Angelo Felice Lopez for bringing to my attention a gap in my original proof of \ref{L2}. I would also like thank the referee for their suggestions, which I found invaluable.

\section{Divisors on $\M_{0,n}$ specified by polynomials}\label{general}\label{basics}
The following diagram is useful in studying divisors on $\ol{M}_{0,n}$: 

\begin{equation}\label{diag1}\begin{CD}
\mb{A}^{n+1}  @<\phi<< \Aff^1[n+1]  @>\chi>> \ol{M}_{0,n+1} @>\psi_r>>  \mb{P}^{n-2}  \\
 @.    @Vp_rVV                                 @ V\pi_r VV        \\
 @.       \Aff^{1}[n]        @>\chi>>                  \ol{M}_{0,n}       
\end{CD}\end{equation}

Above, $\psi_r$ is the Kapranov morphism in index $r$. A Kapranov morphism $\psi_{r}\colon \M_{0,k} \ra \Pj^{k-3}$ for $1 \leq i \leq n$ is constructed by fixing $k-1$ points in general position in $\mb{P}^{k-3}$, and labeling the points $p_t$ for $t \in \{1, \dots, k\} \sm \{r\}$. The relevant fact for our purposes is that given $I \subset \{1,\dots,k\} \sm \{r\}$, the image of $\delta_{I \cup \{r \}}$ under $\psi_r$ is the linear span $\langle p_t \rangle_{t \in I}$. For $|I| \leq k-4$, $\psi_r$ contracts the divisor $\delta_{I \cup \{r\}}$; these are the only exceptional divisors of $\psi_r$. This gives the choice of free generators for $\Cl(\M_{0,k})$ discussed in \ref{intro}, namely the classes of boundary divisors  $E_I^r:= \delta_{I \cup \{r\}}$ for $1 \leq |I| \leq k-4$ and of $H=\psi_r^{-1}(h)$ for $h$ a hyperplane in $\Pj^{n-3}$.  We refer to the free generating set $\langle E_I^r,H\rangle$ obtained via the map $\psi_r$ as the Kapranov basis in index $r$, index $r$ Kapranov basis, or $r$-th Kapranov basis. When the ``special" index is clear, we omit the superscript.

 The map $\pi_r$ is the forgetful morphism in index $r$: drop the $r^{\text{th}}$ marking on a stable rational curve, and stabilize if necessary. \par

The space $\Aff^1[n+1]$ is Fulton-Macpherson configuration space over $\Aff^1$, a partial compactification of the space parameterizing $n+1$ distinct marked points in $\mb{A}^{1}$. The map $\phi$ is an iterated blow-up of $\mb{A}^{n+1}$ along partial diagonals which defines $\mb{A}^{1}[n+1]$. This gives a basis of  $\Cl(\Aff^1[n+1])$ comprised of exceptional divisors $\Delta_I$ over partial diagonals $\Diag_I:=\{x_i = x_j \col i,j \in I \}$ for $3 \leq |I| \leq n+1$ \cite{FM}. A general element of an exceptional divisor $\Delta_I$ consists of a copy of $\Aff^1$ containing marked points in $\{1,\dots, n\} \sm I$, with a nodal tree of $\Pj^1$'s containing the marked points in $I$ attached. \par

As discussed in \cite[p. 195]{FM},  we have a map $p_{r}\colon \Aff^{1}[n+1] \ra \Aff^{1}[n]$ which drops the $r^\text{th}$ marking on an element of $\Aff^1[n+1]$ (analogous to the forgetful map $\pi_{r}\colon\M_{0,n+1} \ra \M_{0,n}$). Moreover, we have a map from $\Aff^1[n+1]$ to $\Pj^1[n+1]$: choose an embedding of $\Aff^1$ into $\Pj^1$ as an affine chart, and this induces a map taking an element of $\Aff^1[n+1]$ to a nodal tree of $\Pj^1$'s. Moreover, we have a map from $\Pj^1[n+1]$ into $\M_{0,n+1}$, mapping a tree of $\Pj^1$'s to its equivalence class modulo automorphisms. A slight obstruction arises because a tree of $\Pj^1$'s in $\Pj^{1}[n+1]$ may not be stable, but this is easily resolved by stabilization. Composition gives the map $\chi$ on the diagram \eqref{diag1}. Commutativity of the middle rectangle is evident from definitions. 

Our goal in this section is to relate divisor classes in the class group of the Fulton--MacPherson space to those in the class group of the moduli space of stable rational curves with marked points. To this end, we compute the class of the pull-backs of boundary divisors from $\M_{0,k}$ under $\chi$. Note that the only boundary divisors contained in $\bp(\delta_I)$ are $\Delta_I$ and $\Delta_{I^c}$. Hence we have that 
\begin{equation}\label{12} \bp(\delta_I) \sim m_1\Delta_I + m_2 \Delta_{I^c}. \end{equation} That $m_1=m_2=1$ is well known, and easy to prove by induction on $n$: %, but for lack of a reference we supply the proof.

\begin{Lemma}\label{m1m2} With maps and definitions as above, $\bp(\delta_I) \sim \Delta_I + \Delta_{I^c}$.
\end{Lemma}

We now return to the set-up of \eqref{diag1}. Given a prime, non-boundary divisor $D \subset \ol{M}_{0,n}$, the divisor $\bp(D)$ is irreducible: $\chi$ has irreducible fibers over $M_{0,n}$ and $\bp(D) = \ol{\bp(D \cap M_{0,n})}$. Moreover, $\bp(D)$ is not an exceptional divisor of $\phi$ since $D$ is non-boundary. Hence $\bp(D)$ is precisely $\ap(\phi(\bp(D)))$, the proper transform of $\phi(\bp(D))$ with respect to $\phi$. \par 

The fact that $\phi(\bp(D))$ is irreducible follows from irreducibility of $\chi\I(D)$, so $\phi(\bp(D))=V(f)$ for some irreducible polynomial $f \in k[x_1, \dots ,x_{n}]$. In this case, we say that the divisor $D$ is {\bf specified by the polynomial $f$}. \par

 Using that $\Aff^1[n]$ is a blow-up of $\Aff^n$ along partial diagonals, we have that $\ap (V(f)) \sim  - \sum k_{J} \Delta_{J}$, where $k_J$ is the multiplicity of $f$ along the partial diagonal $\Diag_J := \{x_i=x_j \, | \, i,j \in J\}$ for $|J| \geq 3$. The next results relate these multiplicities, which are easily computed when $f$ is known, to the class of $D$ with respect to Kapranov bases.

\begin{Theorem}\label{polynomialclasses} Let $\pi_{n+1}\colon \M_{0,n+1} \ra \M_{0,n}$ denote the forgetful morphism in index $n+1$.  Given an irreducible polynomial $f \in k[x_1,\dots,x_n]$ specifying a divisor $D$ on $\ol{M}_{0,n}$ as described above, we have that $$\pi_{n+1}^{-1}(D) \sim  dH-\suml_{\substack{I \subset \{1, \dots, n\} \\ 1 \leq |I| \leq n-3}} m_I E_I^{n+1},$$  where $d=\deg(f)$ and $m_I$ is the multiplicity of f along the complementary partial diagonal $\Diag_{ I^c}$. \end{Theorem}
\begin{Remark} We compute the class of $\pi_{n+1}^{-1}(D) \subset \M_{0,n+1}$ with respect to the $n+1$ Kapranov basis to preserve symmetry. Note that if $f \in k[x_1, \dots, x_n]$ specifies $D$, then the same polynomial viewed as an element of $k[x_1, \dots, x_{n+1}]$ specifies $\pi_{n+1}\I(D)$. In \ref{n+1ton} we explain how to convert the class of $\pi_{n+1}^{-1}(D) \subset \M_{0,n+1}$ to the class of $D \subset \M_{0,n}$ with respect to the index $r$ Kapranov basis.
\end{Remark}
\begin{proof}[Proof of \ref{polynomialclasses}] Define $N = \{1,\dots,n\}, \,\, N_2= \{1,\dots,n-2 \}$. Take $H$ as the pull-back of the linear span $\langle p_i \, | i \in N_2 \rangle$ under $\psi_{n+1}$. Throughout this proof, we let $E_I = E_I^{n+1}$. Using our free generators $\langle E_I, H \rangle$ for $\Cl(\M_{0,n+1})$ and $\langle \Delta_I \rangle$ for $\Aff^1[n+1]$, \ref{m1m2} implies that
\begin{equation}\label{basic1} \chi^{-1}(E_{I}) \sim \Delta_{I \cup \{n+1\}} + \Delta_{N - I}. \end{equation} By \cite[3.4]{KT},
$$ \chi^{-1}(H) \sim  \! \! \! \! \sum_{ \0 \neq J \subset N_{2}} \! \! \! \!  \chi^{-1}(\delta_{ J \cup \{n+1\} })\sim  \sum_{\0 \neq J \subset N_{2}} \! \! \! \! ( \Delta_{J \cup \{ n+1 \} } + \Delta_{N-J} ) $$
\begin{equation}\label{basic2}= \sum_{J \subset N_{2}, |J|>1} \! \! \! \! \!  \Delta_{J \cup \{ n+1 \} } + \! \!  \sum_{ J \subsetneq N_{2}, |J| \geq 1}\! \! \! \! \!  \Delta_{N-J} + \sum_{i \in N_{2}}  \!  \Delta_{\{i,n+1\}} + \Delta_{\{n-1,n\}} .\end{equation}

In \eqref{basic2}, the last terms are those involving divisor classes over partial diagonals of codimension $1$ in $\Aff^{n+1}$. These classes must be expressed in terms of our free generators. Using the relation \begin{equation}\label{ax} \Delta_{\{ \alpha,\beta \}} \sim -\! \! \! \suml_{ \substack{\\ \{ \alpha,\beta \} \subsetneq I }} \!   \Delta_{I}\end{equation} which follows from \cite[p.~184]{FM}, we see that

$$ \Delta_{\{a,n+1\}} \sim - \! \! \! \sum_{a \subsetneq J \subset N} \! \! \Delta_{J \cup \{ n+1 \} } $$ and

$$\Delta_{\{n-1,n\}} \! \sim  \,  - \! \! \! \! \! \sum_{ \{n-1,n \} \subsetneq I } \! \! \! \Delta_{I}.$$  Substituting these into  \eqref{basic2} yields

 \begin{equation} \beta^{-1}(H) \sim  \! \!  \! \! \! \sum_{J \subsetneq N_{2}, |J| \geq 1}\! \! \! \! \! \! \!  \Delta_{N-J} + \Delta_{n-1,n} + \Omega  \end{equation}
\begin{equation}\label{planeclass2}  \sim \! \! \! \sum_{ J \subsetneq N_{2}, |J| \geq 1} \! \! \! \! \! \!  \Delta_{N-J} \, \,  - \! \! \! \!  \sum_{ \{n-1,n \} \subsetneq I \subset N } \! \! \! \! \! \Delta_{I} \, + \,  \Omega,\end{equation} where $\Omega$ denotes a sum of free generators $\Delta_I$ with $n+1 \in I$. We subsequently redefine $\Omega$ to absorb such terms, which turn out to be superfluous. Summing over $N \sm J$ for $J \subsetneq N_{2}$ and $|J|\geq1$ is equivalent to summing over $I \subsetneq N $ with $\{n-1,n\} \subsetneq I$. Returning to \eqref{planeclass2} we obtain

$$ \beta^{-1}(H) = \sum_{\{n-1,n\} \subsetneq I \subsetneq N} \! \!  \! \! \! \! \! \! \!  \Delta_{I} -  \sum_{ \{n-1,n \} \subsetneq I \subset N}\! \! \! \! \! \!  \! \! \Delta_{I} \, + \, \Omega = -\Delta_{N} \,  + \,  \Omega.$$ 

We can now compute the class $\beta^{-1}(D)$:

$$\beta^{-1}(D) \sim \beta^{-1}( dH - \! \! \! \! \! \! \! \! \! \sum_{\substack{I \subset N\\ 1 \leq |I|\leq n-3}} \! \! \! \! \! \! \! \! \! m_{I}E_{I}) $$ 
$$\sim -d \, \Delta_{N} -\! \! \! \! \! \! \sum_{\substack{I \subset N\\ 1\leq|I|\leq n-3}}\! \! \! \! \! \! m_{I}  (\Delta_{I \cup \{n+1\}} + \Delta_{N - I}) \,\, + \,\, \Omega$$ 

\begin{equation} \label{final} \beta^{-1}(D) \sim  -d \, \Delta_{N}- \!\!  \sum_{\substack{I \subset N \\ 1<|I|<n-2} } \!\!\! \! m_I \Delta_{N - I}  \,  +  \, \Omega. \end{equation} 

 For $I \subset N$ satisfying $2\leq |I|\leq n-3$, we have a single term in \eqref{final}  involving the free generator $\Delta_{N-I}$, with coefficient $m_I$. Hence $m_I = k_{N-I}$. \par

It remains to determine the coefficient of $H$. The above analysis shows that we have a single summand $d \Delta_N$ in the class of $\beta^{-1}(H)$, and the proper transforms of boundaries contribute no multiples of $\Delta_N$ to the sum. Hence the multiplicity of $f$ along the diagonal $\Diag_N$ is $d$. We claim that if $D \subset \M_{0,n+1}$ is an irreducible non-boundary divisor and $f \in k[x_1, \dots,x_n]$ satisfies $V(f)=\phi(\bp(D))$, then $f$ is a homogeneous polynomial. Furthermore, for some $g \in k[x_1, \dots ,x_n]$ we have that $$f(x_1,\dots,x_n)=g(x_1-x_2,x_2-x_3,\dots,x_{n-1}-x_{n}).$$ This follows from the fact that $V(f) \cap \Aff^{n+1}\sm\{ \text{ diagonals } \}$ is stable under affine transformations, in particular rescaling and translation. \par

Consequently, substituting $x_i \mapsto (x_i+t)$ for $1 \leq i \leq n$ to compute the multiplicity along $\Diag_N$ leaves $f$ invariant. Since the polynomial is homogeneous we have that the multiplicity of $f$ along the partial diagonal $\Diag_N$ is  precisely the degree of $f$, as was to be shown.  \end{proof}

We now introduce notation to facilitate comparison of the class of $D \subset \M_{0,n}$ and that of $\pi_{n+1}^{-1}(D) \subset \M_{0,n+1}$. Let $d_I$ and $\delta_I$ denote the boundary divisors on $\M_{0,n}$ and $\M_{0,n+1}$, respectively.  For  $r \in \{1,\dots,n\}$ and $I \subset \{1, \dots,n\} \sm \{r\}$ with $1 \leq |I| \leq n-4$, let $e_I^r=d_{I \cup \{r\}}$ and let $h \subset \M_{0,n}$ denote the pull-back of a hyperplane under the Kapranov morphism in index $r$. Let $E_I^{n+1}=\delta_{I \cup \{n+1\}}$ and let $H$ be the pull-back of a hyperplane under $\psi_{n+1} \colon \M_{0,n+1} \ra \Pj^{n-2}$.

\begin{Proposition}\label{n+1ton} Let $D \subset \M_{0,n}$ be an irreducible divisor. Suppose that

$$ \pi_{n+1}\I(D) \sim aH-\suml_{1 \leq |I| \leq n-3}m_I E_I$$ on $\overline{M}_{0,n+1},$ where $\pi_{n+1} \colon \M_{0,n+1} \ra \M_{0,n}$ is the forgetful morphism in index $n+1$. Then 

$$ D \sim m_{\{r\}}h-\suml_{\substack{2 \leq |I| \leq n-3 \\ r \in I }} m_I e_{ I \sm \{r\}}^r$$
as a divisor on $\M_{0,n}$, with notation as in the paragraph preceding the result. 
\end{Proposition}
\begin{proof} For concreteness, assume $r=1$. The argument centers on computing classes of pull-backs of free generators $e_I$ and $h$ under $\pi_{n+1}$. The proposition is a straightforward calculation which appeals to three basic facts:
\begin{enumerate}
\item[i.] $ \pi_{n+1}^{-1} (d_J) \sim \delta_{J \cup \{n+1\}} + \delta_{J}. $
\item[ii.] $h \sim \suml_{\substack{a,s \in F \\1 \notin F}} d_F,$ for any $a,s$ distinct in $\{2, \dots, n\}$.
\item[iii.] $\delta_{\{i,j\}} \sim H - \suml_{\substack{i,j \notin F \\ 2 \leq |F| \leq n-3 }} E_F$ for $i <j \in \{1,\dots,n\}.$ 

\end{enumerate}

As previously discussed, (i) follows from noting that $\pi_{n+1}$ has reduced fibers; (ii) is proved in \cite[\S3.4]{KT}; (iii) is a reformulation of (ii) applied to divisors on $\M_{0,n+1}.$
Now consider $e_I = \delta_{I\cup \{1\}}$ for $I \subset \{2,\dots,n\}$ with $1 \leq |I| \leq n-4$. In the case that $2 \leq |I| \leq n-4$, we have

\begin{equation}\label{In=1}\pi_{n+1}^{-1}(d_{I \cup \{1\}}) \sim \delta_{I \cup \{1,n+1\}} + \delta_{I \cup \{1\}} \sim E_{I \cup \{1\}}^{n+1} + E_{\{2,\dots,n\} \sm I}^{n+1},\end{equation} appealing to (i). If $I = \{i\}$, then    

\begin{equation}\label{I=1} \pi_{n+1}^{-1}(d_{\{i,1\}}) \sim \delta_{ \{i,1,n+1\}} + \delta_{\{i,1\}} \sim  E_{\{1,i\}}^{n+1} + H - \suml_{\substack{1,i \notin J \\ 2 \leq |J| \leq n-3}}E_J^{n+1},\end{equation} using (iii) for the last equivalence. Last, we compute

$$ \pi_{n+1}^{-1}(h) \sim \suml_{\substack{ a,s \in F \\  1 \notin F }}\pi_{n+1}^{-1}(d_F).$$ Applying (i), we obtain 

\begin{equation}\label{intermediate1} \pi_{n+1}^{-1}(h) \sim \suml_{\substack{ a,s \in F \\  1 \notin F \\ 2 \leq |F| \leq n-3}} \delta_{F \cup \{n+1\}} + \suml_{\substack{ a,s \in F \\  1 \notin F \\ |F|=n-2}} \delta_{F\cup\{n+1\}} + \suml_{\substack{ a,s \in F \\  1,n+1 \notin F \\ 2 \leq |F| \leq n-2 }} \delta_F\end{equation}

\begin{equation}\label{intermediate2} \sim \!\!\!\!\! \suml_{\substack{ a,s \in F \\  1 \notin F \\ 2 \leq |F| \leq n-3}}\!\!\!\!\! \delta_{F \cup \{n+1\}} +\!\!\!\!\! \suml_{\substack{ a,s \in F \\  1 \notin F \\ |F|=n-2}}\!\! \delta_{F\cup\{n+1\}} +\!\! \suml_{\substack{ a,s \in F \\  n+1 \notin F \\ 2 \leq |F| \leq n-1 }}\!\! \delta_F -\!\! \suml_{\substack{a,s, 1 \in F\\ n+1 \notin F \\ 2 \leq |F| \leq n-1}}\!\! \delta_F -\delta_{ \{2,\dots,n\}}.\end{equation} Note that the term involving $\delta_{ \{2,\dots,n\}} = \delta_{ \{1,n+1\}}=E_{\{1\}}$ in \eqref{intermediate2} must be subtracted because the last term in \eqref{intermediate1} includes only $\delta_F$ for $|F|\leq n-2$. 

\par Using (ii), \eqref{intermediate2} can be rewritten as

$$ \suml_{\substack{ a,s \in F \\ |F| \ge 2\\ 1 \notin F}} E^{n+1}_{F} + \suml_{\substack{ i \notin \{a,s,1,n+1\}}} (H - \suml_{\substack{1,i \notin J \\ 2 \leq |J| \leq n-3}} E^{n+1}_J ) + H - \suml_{\substack{2\leq|F|\leq n-3 \\ a,s,1 \notin F}}E^{n+1}_F-E^{n+1}_{\{1\}}$$

\begin{equation}\label{planekey} = -E^{n+1}_{\{1\}}+ \Omega, \end{equation} where $\Omega$ absorbs terms proportional to $H$ or $E_J$ for $1 \notin J$. Using \eqref{I=1},\eqref{In=1}, and \eqref{planekey}, we see that 

$$ \pi_{n+1}^{-1}\big(bh-\sum_{1 \leq |I| \leq n-4\}} k_I e_I \big)  \sim -b E^{n+1}_{\{1\}} - \suml_{\substack{I \subset \{1,\dots n\} \\ 2 \leq |I| \leq n-4}} k_I E^{n+1}_{I \cup \{1\}} + \Omega,$$ where again terms in $\Omega$ are linearly independent of those explicitly written. Hence we have $b = m_{\{1\}}$ and $k_I = m_{I \cup \{1\}}$ as was to be shown. \end{proof}

\begin{Corollary} \label{polyn+1ton} If $D\subset \M_{0,n}$ is specified by  $f \in k[x_1,\dots,x_n]$ as in \ref{polynomialclasses}, the class of $D$ in the index $r$ Kapranov basis for $\Cl(\M_{0,n})$ is $$m_{\{r\}}H-\sum_{1 \leq |I| \leq n-3  } m_{I\cup \{r\}} E_I,$$ where $m_J$ is the multiplicity of $V(f)$ along the partial diagonal $\Diag_{\{1,\dots,n\} \sm J}$.\end{Corollary}

The following Macaulay2 code uses the formulae derived above to give the class of a divisor specified by a polynomial equation with respect to the Kapranov basis in index $n+1$. It is important to note that result of this calculation is actually the divisor class modulo a large prime. For small $n$ and low-degree polynomials, this is unlikely to result in discrepancies with the actual class. Moreover, the code is best suited for experimentation and motivation; in this context, sufficient certainty about a given class can be obtained by varying the modulus.  \par

To implement the code, first import the code into Macaulay2. Then define a polynomial $f=f(x_1,\dots,x_n)$. The command $T(f)$ outputs the class encoded as a polynomial as follows: the class $H$ is represented by a variable $z$, and the classes $E_I$ are represented as a monomials $\prod_{i \in I} x_i$. \par

A brief explanation of the code: the first part creates an $n \times 2^n$ binary matrix $V$ encoding partial diagonals. The diagonal $\{x_i = x_j \col i,j \in I\}$ corresponds to the row with $1$'s in the rows corresponding to indices in $I$, and zeroes elsewhere. The associated matrix $W$ omits partial diagonals along which multiplicities need not be calculated. \par
Using this matrix $W$, the second part of the code defines functions (taking as input polynomials) which are composed to calculate the multiplicity along relevant diagonals. More explicitly, the code first performes a change of variables and then calculates the degree of the resulting polynomial, viewed as a polynomial in the variable $t$.  \par

{\small
\begin{verbatim} 
--before running code, choose n between 6 and 10.
n=6;
R = ZZ/21977[x_0...x_9,b_0...b_9,z,
Degrees=>{-1,-1,-1,-1,-1,-1,-1,-1,-1,-1,0,0,0,0,0,0,0,0,0,0,0}]; 
F = (i,j) -> if i==j then 1 else 0;
--code is for divisor in M_{0,n} specified by a polynomial in n variables. 
--output is class of pull-back in M_{0,n+1} with respect to Kapranov in index n+1
--need multiplicities of polynomial along partial diagonals.
--the following encodes diagonals in a matrix.
u= matrix table(1,2^n,(i,j)-> if j<2^(n-1) then 1 else 0); 
V = matrix table(n,2^n,(i,j)->u_(0,(2^i*j)%(2^n)) );
W = matrix table(n,2^n, (i,j)-> 
  if sum(for i from 0 to n-1 list F(1,V_(i,j)))==n or 
  sum(for i from 0 to n-1 list F(1,V_(i,j)))<3 then 0 else V_(i,j));
--next make substitutions along the diagonals
--given a polynomial f, (Y(f)) is a matrix with each column encoding 
--a diagonal in the first n entries and multiplicity in the n+1st. 
--BB(LL(Y(f))) encodes the class as a polynomial. 
--E_I = monomial that is product of x_i's for i in I. 
g = (i,j) -> if i<n and sum(for l from 0 to n-1 list W_(l,j)) != 0
	then ( F(0,W_j_i)*b_i + F(1,W_j_i)*(z) + x_i ) else 0;
h = (l,P) -> sub(P,{x_0=>g(0,l), x_1=>g(1,l), 
  x_2=>g(2,l), x_3=>g(3,l), x_4=>g(4,l), 
  x_5=>g(5,l), x_6=>g(6,l),x_7=>g(7,l),
  x_8=>g(8,l),x_9=>g(9,l)});
Y = P -> for i from 0 to 2^n-1 list matrix table(n+1,1,(j,l) -> 
  if h(i,P)==0 or first degree(h(i,P))==0 
  then 0 else if j==n then (first degree(h(i,P))) 
  else F(0,W_(j,i) )*(x_(j)) );
a = v-> if v_(n,0) == 0 then 0 else 
  product(flatten(entries((compress transpose v))));
LL = Y -> apply(Y,a);
BB= LL -> sum LL;
T = P-> BB(LL(Y(P)))-first degree(P)*z;
\end{verbatim} }

%----------------- HYPERTREES --------------------------------------------------------------------------------------------------------------------------------

\section{Equations of hypertree divisors}\label{ht}
The following definitions are from \cite{CT}. A \B{hypertree} on a set $N$ is a collection $\Gamma=\{\Gamma_1,\dots,\Gamma_d\}$ of subsets of $N$ satisfying:
\begin{enumerate}
\item For any $j \in \{1,\dots,d\}$, $|\Gamma_j| \ge 3.$
\item  Each $i\in N$ is contained in at least two distinct $\Gamma_j$'s. 
\item Convexity: $ \bigr| \bigcup\limits_{j\in S} \Gamma_j \bigl| -2 \ge\suml_{j\in S}(|\Gamma_j|-2)$ for any $S\subset \{1,\dots,d \}.$
\item  Normalization: $ |N|-2=\suml_{1\leq j \leq d}(|\Gamma_j|-2).$\\
\end{enumerate} 
$\Gamma$ is \B {irreducible} if the convexity condition (3) is strict for $1<|S|<d$. A \B{ planar realization} of a hypertree $\Gamma = \{ \Gamma_1, \dots, \Gamma_d\}$ is a collection of points $p_1,\dots,p_n \in \mb{P}^2$ satisfying that $p_i,p_j,p_k$ are collinear if and only if there exists an  $\alpha \in \{1,\dots,d\}$ such that $i,j,k \in \Gamma_{\alpha}$.

\begin{figure}[h]\label{planarreal}
\center
\includegraphics[scale=.35]{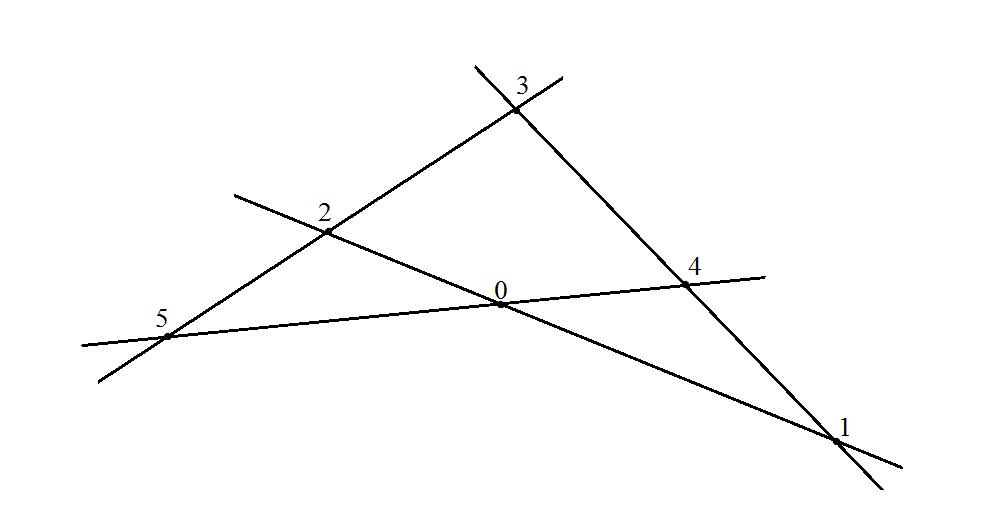} 
\caption{Planar realization of the complete quadrilateral, defined by $\Gamma = \{012,314,045,325\}$.}
\end{figure}

Given planar realization, the images of $p_1, \dots, p_n$ under projection from a general point give $n$ distinct marked points on $\Pj^1$. Given $\Gamma$, define the \B{hypertree divisor} $D_{\Gamma} \subset \ol{M}_{0,n}$ as the closure of the locus 
 $$\big\{ [\mb{P}^1;q_1,\dots,q_n] \col \E \text{ a realization $\{p_i\}$ and projection } \pi \text{ with } q_i = \pi (p_i) \big\}.$$ For $\Gamma$ irreducible, Castravet and Tevelev show that $D_{\Gamma}$ is a nonempty irreducible divisor generating an extremal ray of $\Eff(\M_{0,n})$. \par 
 
 Rather than defining $D_\Gamma$ as above, one might consider the closure of the locus of equivalence classes $[\mb{P}^1;q_1,\dots,q_n]$ such that $q_i$'s are projections of points $\{p_1,\dots,p_n\} \subset \mb{P}^2$ where  $p_i, p_j, p_k$ are collinear if $i,j,k \subset \Gamma_\alpha$, and not all $p_i$ are collinear. The distinction here is that we no longer require ``only if". It is nontrivial that this weaker definition coincides with that of $D_\Gamma$, and is proved in \cite[\S4]{CT}. We will use this characterization to obtain equations in $k[x_1,\dots,x_n]$ specifying irreducible hypertree divisors (where this specification is in the precise sense discussed in \ref{basics}). Our proof is a direct generalization of results in [CT] for the case where all subsets comprising the hypertree have three elements. \par
We first set up some notation. Given a subset $\Gamma_i = \{ a_{i1}, \dots,a_{ik_i} \}$, let $\Gamma_{ij} = \{ a_{i1}, a_{i2},a_{ij} \}$ for $3 \leq j \leq k_i-2$. By normalization $$\sum_{i=1}^{d}|\Gamma_i|=\sum_{i=1}^{d}(k_i -2) = n-2,$$ and from each $\Gamma_i$ we define precisely $k_i-2$ sets $\Gamma_{ij}$, so the total number of subsets $ \Gamma_{ij}$ for $1 \leq i \leq d, 3 \leq j \leq k_i,$ is $n-2$. Let $\{ G_\alpha \}_{1 \leq \alpha \leq n-2}$ to be an ordering of the collection of $\Gamma_{ij}$'s. With this, we can state the following

\begin{Theorem}\label{htpoly} Let $\Gamma = \{ \Gamma_1, \dots , \Gamma_d \}$ be a hypertree. With notation as preceding the theorem, define an $(n-2) \times n$ matrix  $\bf{A}$ by
$$ G_\alpha = \{ i,j,k \} \Lra \bd{A}_{\alpha,i} = (x_j-x_k), \, \bd{A}_{\alpha,j}=(x_k-x_i), \, \bd{A}_{\alpha,k}=(x_i-x_j).$$ If $\beta \notin G_\alpha,$ then let $\bd{A}_{\alpha.\beta}=0$. 
Define $\bf{B}$ as the $(n-3) \times (n-3)$ matrix obtained from $\bf{A}$ by deleting a row and all columns in which the entries of that row are nonzero. The hypertree divisor $D_\Gamma$ is specified by 
$$\frac{\det \bf{B}}{\prod_{i=1}^{d}(x_{a_{i1}}-x_{a_{i2}})^{k_i-3}}.$$
\end{Theorem}

\begin{proof} The condition that points  $x_1,\dots x_n \in \mathbb{A}^1$ can be obtained from the projection of a hypertree curve is equivalent to existence of $y_1,\dots,y_n \in \mb A^n$ so that, defining $p_i = (x_i,y_i)$, the following is satisfied:  \begin{equation}\label{***} \text{ Not all $p_k$ are collinear, and } i,j,w \in \Gamma_ k \Lra p_i,p_j,p_w \text{ collinear}. \end{equation}

\par By construction, $p_i, p_j,p_w$ are collinear whenever $i,j,w \in \Gamma_k$ for some $k$ if and only if $p_x, p_y,p_z$ are collinear  whenever $x,y,z \in G_i$ for some $i$. We apply the argument given in \cite[\S8]{CT} to the subsets $G_i$ to obtain $\bf{A}$ as defined above so that a solution to $\bd{A}(y_1,\dots y_n)^{T} = 0$ with not all points $p_i$ collinear implies that $[\mb{P}^1; a_1,\dots,a_n] \in D_\Gamma$, if $[\Pj^1:a_1,\dots,a_n] = \chi(x_1,\dots, x_n)$, where $\chi\: \Aff^1[n] \lra \M_{0,n}$ is as defined in \ref{intro}. \par

If a solution $y = (y_1,\dots,y_n)^{T}$ to $\bd{A}y = 0$ exists, we may choose coordinates so that three points corresponding to the indices in some fixed $G_{i_0}$ lie along $y=0$. We shall subsequently refer to this $G_{i_0}$ as a {\bf pivot subset}. Requiring that not all $p_i$'s are collinear and setting $y_i=0$ for $i \in G_{i_0}$, we seek a nontrivial solution $\bd{B} y=0$, where $\bd{B}$ is as defined in the theorem. 

For points $x_1,\dots,x_n$ there exists a configuration of points $p_1,\dots,p_n$ satisfying \eqref{***} if and only if $\det {\bf B} (x_1,\dots,x_n) =0$.  Let $A$ denote $\mb{A}^n$ minus partial diagonals of codimension greater than 1; what we have shown is that $\phi_{*}^{-1}(V(\det {\bf B}) \cap A)=\chi^{-1}(D_\Gamma \cap M_{0,n})$, with maps $\phi$ and $\chi$ as defined in \ref{basics}. Hence $\det \bd{B}$ is the correct equation for $D_\Gamma$ on $M_{0,n}$, but $\det \bd{B}$ may include erroneous boundary factors corresponding to partial diagonals.

\begin{Claim} For each $\Gamma_i = \{ a_{i,1},a_{i,2}, \dots, a_{i,k_\alpha} \}$, $\max \{m : (x_{a_{i,1}}-x_{a_{i,2}})^m| \det {\bf B} \} = k_\alpha-3$. \end{Claim} 

 Given the claim, $D_\Gamma$ is specified  by \begin{equation} \label{htpolynomial} g(x_1,\dots,x_n):=\frac{\det {\bf B}}{\prodl_{i=1}^{d} (x_{a_{i,1}}-x_{a_{i,2}})^{k_\alpha-3}}.\end{equation}  To see that $\ap(V(g))=\bp(D_\Gamma)$, note that $\deg(g)=n-2-\sum_{i=1}^{d}(k_i-2)=d-1$, where the last equality invokes normalization of $\Gamma$. By \cite[\S4.2]{CT}, we know that $\pi_{n+1}\I(D_\Gamma) \sim (d-1)H + \dots \, \, $. From \eqref{polynomialclasses}, a divisor $D$ specified by a polynomial $F$ satisfies $\pi_{n+1}\I(D) \sim aH + \dots \, $ where $a= \deg(F)$. Hence degree considerations show $\phi_{*}^{-1}(V(g))= \chi^{-1}(D_\Gamma)$. 

\par We now prove the claim. Consider the rows of $\bd{B}$ corresponding to a given subset $\Gamma_i=\{a_{i1},\dots,a_{ik_i}\}$. Assume for simplicity $i = 1$ and $a_{i1},a_{i2},\dots,a_{ik_i}=1,2,\dots,k_1$; the argument generalizes. The first $k_1-2$ rows of $\mathbf{A}$ are as follows:
$$ \begin{pmatrix}
  x_2-x_3 & x_3-x_1 & x_1-x_2 & 0 & 0 & \cdots & 0 & \cdots \\
  x_2-x_4 & x_4-x_1 & 0&  x_1-x_2 & 0 & \cdots & 0 & \cdots \\
  x_2-x_5 & x_5-x_1 & 0 &  0 & x_1-x_2 & \cdots & 0 & \cdots \\
\vdots & \vdots & \vdots & \vdots & \vdots & \ddots & \vdots & \ddots \\ 
 x_2-x_{k_i} & x_{k_i}-x_1 & 0&  0 & 0 & \cdots & x_1-x_2 & \cdots \\
 \end{pmatrix} $$

\par In passing from the matrix $\bd{A}$ to $\bd{B}$, the rows of $\bd{A}$ shown above can be altered in three ways. Let $G_{i_0}$ be the pivot subset used to obtain $\bd{B}$.

\begin{enumerate}
\item $ G_{i_0} \subset \Gamma_1$. Without loss of generality we may assume $G_{i_0}$ corresponds to the first row of $\bd{A}$. The first $k_1-3$ rows of $\bd{B}$ then appear as follows:
$$ \begin{pmatrix}
  x_1-x_2 & 0 & \cdots & 0 & \cdots \\
  0 & x_1-x_2 & \cdots & 0 & \cdots \\
 \vdots & \vdots & \vdots & \vdots & \ddots  \\ 
  0 & 0 & \cdots & x_1-x_2 & \cdots \\
 \end{pmatrix} $$
Evidently $(x_1-x_2)^{k_1-3}$ divides $g$.

\item $ |G_{i_0} \cap \Gamma_1| = 0 $. In this case, the first $k_1-2$ rows of $\bd{B}$ will be identical to those of $\bd{A}$ given above. Adding column $2$ to column $1$ gives 
$$ \begin{pmatrix}
  x_2-x_1 & x_3-x_1 & x_1-x_2 & 0 & 0 & \cdots & 0 & \cdots \\
  x_2-x_1 & x_4-x_1 & 0&  x_1-x_2 & 0 & \cdots & 0 & \cdots \\
  x_2-x_1 & x_5-x_1 & 0 &  0 & x_1-x_2 & \cdots & 0 & \cdots \\
\vdots & \vdots & \vdots & \vdots & \vdots & \ddots & \vdots & \ddots \\ 
 x_2-x_1 & x_{k_i}-x_1 & 0&  0 & 0 & \cdots & x_1-x_2 & \cdots \\
 \end{pmatrix} $$
Expansion  across rows shows $(x_1-x_2)^{k_1-3}$ divides the $\det(\bf{B})$.

\item $ |G_{i_0} \cap \Gamma_1|=1$. This is the situation where precisely one column and no rows of the submatrix of $\bd{A}$ corresponding to $\Gamma_1$ are removed in passing to $\bd{B}$. Let $\{ h \} = G_{i_0}\cap \Gamma_1$. We have two subcases to consider: 
\begin{itemize} 
\item $3 \leq h$. This results in a submatrix of the first $k_1-2$ rows of $\bd{B}$ of the form 

$$ \begin{pmatrix}
  x_2-x_3 & x_3-x_1 & 0 & 0 & \cdots & 0 & \cdots \\
  x_2-x_4 & x_4-x_1 &  x_1-x_2 & 0 & \cdots & 0 & \cdots \\
  x_2-x_5 & x_5-x_1 &  0 & x_1-x_2 & \cdots & 0 & \cdots \\
\vdots & \vdots & \vdots & \vdots & \ddots & \vdots & \ddots \\ 
 x_2-x_{k_i} & x_{k_i}-x_1 &  0 & 0 & \cdots & x_1-x_2 & \cdots \\
 \end{pmatrix} $$
The argument from case 2 goes through (with minor adjustments) to show $(x_1-x_2)^{k_1-3}$ is a factor of $\det \bf{B}$.

\item $h =1$ or $h=2$.
\par This results in the first $k_1-2$ rows of $\bd{B}$ of the form 
$$ \begin{pmatrix}
  x_3-x_1 & x_1-x_2 & 0 & 0 & \cdots & 0 & \cdots \\
  x_4-x_1 & 0&  x_1-x_2 & 0 & \cdots & 0 & \cdots \\
   x_5-x_1 & 0 &  0 & x_1-x_2 & \cdots & 0 & \cdots \\
 \vdots & \vdots & \vdots & \vdots & \ddots & \vdots & \ddots \\ 
 x_{k_i}-x_1 & 0&  0 & 0 & \cdots & x_1-x_2 & \cdots \\
 \end{pmatrix} $$
 Evidently $(x_1-x_2)^{k_1-3} $ divides $\det \bd{B}$.
\end{itemize}

\end{enumerate}

This proves the claim.  \end{proof}

All hypertrees up to permutation for at most $11$ vertices were found in \cite{Sch}.  Enumeration of small irreducible hypertrees is as follows: 1 for 6 or 7 vertices; 3 for 8 vertices; 11 for 9 vertices; and 96 for 10 vertices. \par

Using our Macaulay2 program for computing classes specified by polynomial equations (see Section \ref{general}) and the polynomial \eqref{htpolynomial}, we  computed all divisor classes corresponding to irreducible hypertrees for $ 6\leq n \leq10$. We additionally wrote a program to compute symmetry group sizes, and computed symmetry groups of irreducible hypertree classes for $6 \leq n \leq 8$. \par 

Particularly nice hypertrees are obtained via even triangulations of a two-sphere: given a bi-colored (say black and white) triangulation of the two-sphere with $n$ vertices, one can consider unordered triplets $\{i,j,k\}$ corresponding to the vertices of black triangles. The collection of all such triplets gives a set of subsets of $\{1, \dots, n\}$; Castravet and Tevelev show that, for any bicolored triangulation, this collection of subsets yields a hypertree. They call hypertrees obtained in this way {\bf spherical hypertrees}. These spherical hypertrees are irreducible unless the triangulation is a connected sum \cite[1.6]{CT}. For $6 \leq n \leq 10$, we classify spherical hypertrees in our database. Spherical hypertrees are further discussed in \ref{sphere}: certain spherical hypertree divisors are seen arise as certain Chen--Coskun divisors. \par

For the complete database and Macaulay2 code specific to hypertree divisors, see  \cite{Op}. It is hoped that these data will prove useful in further investigations of hypertrees and other divisors. In addition to the production of the database, the code from the previous section was applied to explore properties of  divisors, motivating the discovery of our counterexamples to \ref{conj}. We provide the counterexample in \ref{counterex}, but must first describe the family of divisors used in the construction.

%----------------------------------- CC DIVISORS  -------------------------------------------------------------------------------------

\section{Chen--Coskun divisors}\label{CCdivs}

In \cite{CC}, Chen and Coskun define divisors on the moduli space $\M_{1,n}$ of genus $1$ curves with $n$ ordered markings as follows. Given an $n$-tuple of integers $$\z=(a_1, \dots, a_n)$$ with $\sum_i a_i=0$, define $D_\z \subset \M_{1,n}$ to be the closure of the locus of smooth genus $1$ curves $[E; p_1, \dots, p_n]$ so that $\sum a_i p_i = 0$ in the Jacobian of the curve. Results on these divisors include that, for $\gcd(a_1, \dots, a_n) =1$ and $n \geq 3$, the divisor $D_\z$ is an irreducible, rigid effective divisor generating an extremal ray of the effective cone of $\M_{1,n}$. Moreover, there are infinitely many of distinct divisors of this form  on $\M_{1,n}$ for each $n \geq 4$, showing that $\Eff(\M_{1,n})$ is not finitely generated \cite{CC}. \par

 The natural clutching morphism $\vphi\colon \ol{M}_{0,n+2} \lra \ol{M}_{1,n}$ identifies marked points $p_{n+1}$ and $p_{n+2}$ on a rational curve in $\ol{M}_{0,n+2}$:\begin{figure}[h]
\center
\includegraphics[scale=.55]{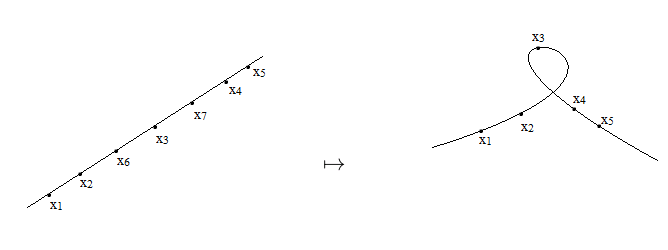}
\caption{The clutching morphism $\vphi \colon \M_{0,7} \ra \M_{1,5}.$}

\end{figure}\newline One might ask what can be said about the proper transforms under $\vphi$ of the divisors defined in \cite{CC}. However, the definition given by Chen and Coskun does not lend itself to study of these proper transforms: the image $\vphi(\M_{0,n+2})$ lies entirely in the complement of the smooth locus, and the definition above is in terms of the closure of a collection of smooth curves. Hence, we give an alternate definition entirely within the locus of nodal genus $1$ curves. \par  For ${\bf a}=(a_1,\dots,a_n) \in \mb{Z}^n$ satisfying $\sum_{i=1}^{n}a_i = 0$ and $\gcd(a_1,\dots,a_n)=1$, define $D_{\z}$ as the closure in $\vphi (\M_{0,n+2})$ of the locus of irreducible nodal curves $ [C; p_1,\dots,p_n] $ with $n$ distinct smooth markings such that $\sum_{i=1}^{n} a_i p_i = 0$ in $\Pic^0(C) \simeq \mb{G}_m$.
It is clear that our subset $D_\z \subset \vphi(\M_{0,n+2})$ is the intersection of the divisor $D_{\z}$ defined in \cite{CC} with $\vphi(\M_{0,n+2})$, but we will not use this fact.  Henceforth, $D_\z$ will refer to our divisor defined on $\vphi(\M_{0,n+2})$ unless otherwise noted.

\begin{Lemma}\label{irredG} The locus $D_\z \subset \vphi(\M_{0,n+2})$ is an irreducible divisor. \end{Lemma}

\begin{proof} Let $Y=\mb{G}_m^n \sm \{ \text{ diagonals } \}$. Consider the following commutative diagram:

$$\begin{CD}
Y @>(\lambda_1, \dots, \lambda_i) \mapsto \prod_{i=1}^{n} \lambda_i^{a_i}>> \mb{G}_m \\
@V \gamma VV @VVwV \\
\vphi(M_{0,n+2}) @>[C\colon p_1, \dots, p_n] \mapsto \sum_{i=1}^n a_i p_i>> \Pic^0(C) 
\end{CD}.$$ \newline  Above, $\gamma$ is induced by an isomorphism of the smooth locus of an irreducible nodal cubic with $\mb{G}_m$, and maps an $n$-tuple of distinct points to their isomorphism class in $\vphi(\M_{0,n+2})\subset \M_{1,n}$.  The map $w$ is the canonical identification of $\Pic^0(C)$ with $\mb{G}_m$. Note that $\gamma$ is surjective onto $\vphi(M_{0,n+2})$. Define $$S = \big\{ (p_1, \dots , p_n) \in Y \col \prod_{i=1}^{n} p_i^{a_i} = 1 \big\},$$ and $\gamma^{-1}(D_\z \cap \vphi(M_{0,n+2}))=S$. Hence it will suffice to show irreducibility of $S$. 

Recall that endomorphisms of $\mb{G}_m^{n}$ are given by $$p_i \mapsto \prod_{j=1}^{n} p_j^{r_{ij}}$$ for $r_{ij} \in \mathbb Z$; so we represent an endomorphism via an integral matrix acting on exponents:

$$ R=\begin{pmatrix}
 r_{11} & r_{12} & \dots & r_{1n} \\
 r_{21} & r_{22} & \dots & r_{2n} \\
 r_{31} & r_{32} & \dots & r_{3n} \\
\vdots & \vdots & \vdots & \vdots & \\
  r_{n1} & r_{n2} & \dots & r_{nn} \\
 \end{pmatrix}. $$ 
The corresponding map is an automorphism if and only if $|\! \det R| = 1$. \par 

Suppose that there is an automorphism $h$ such that $$p_1 \xmapsto{h} \prod_{i=1}^{n}p_i^{a_i}.$$ Then $$\{(p_1,\dots,p_n) \col  p_1^{a_1}\dots p_n^{a_n}=1\} \simeq \{(q_1, \dots , q_n) \col q_1=1\},$$ where the isomorphism is induced by the given endomorphism. This is the graph of a morphism from $\mb{G}_m^{n-1}$ to $\mb{G}_m$, hence an irreducible divisor. So it suffices to show that there exists an automorphism with matrix $R$ such that $r_{i1}=a_i$ for $1 \leq i \leq n$. We show this by induction on $n$. For $n=2$, the condition that $\gcd(a_1,a_2)=1$ gives that there exist $c_1$, $c_2$ such that $a_1c_1-a_2c_2=1$. A matrix with the desired property is then given by

$$ R=\begin{pmatrix}
 a_{1} & c_2 \\
a_2 & c_1 \\
\end{pmatrix} .$$\\
Now consider the case for $S \subset \mb{G}_m^{k+1}$ with $\gcd(a_1,\dots, a_{k+1})=1$. Let $s:=\gcd(a_1, \dots,a_k)$. Factoring out the  $\gcd$, inductively there is an automorphism of $\theta$ of $\mb{G}_m^k$ taking $p_1^{a_1} \dots p_k^{a_k} \xmapsto{\theta} q_1^s$, where $q_i := \theta(p_i)$. The map $\theta$ extends to an automorphism of $\mb{G}_m^{k+1}$ with $p_{k+1} \xmapsto{\theta} p_{k+1}$, and we have 
 $$S \simeq \{(q_1, \dots, q_{k+1}) \col q_1^{s}q_{k+1}^{a_{k+1}}=1\}.$$ 

The assumption that $\gcd(a_1,\dots,a_{k+1})=1$ forces $\gcd(s,a_{k+1})=1$. Hence the induction is completed by applying the $k=2$ case to obtain an appropriate automorphism of $\langle q_1, q_{k+1} \rangle \simeq \mb{G}_m^{2}$.\end{proof}

We now give explicit formulas for the class of the proper transform of $D_{\z}$ with respect to the clutching morphism $\vphi$.

\begin{Theorem}\label{classeqs}   Given $\z=(a_1,\dots,a_n)$ with $\sum_i a_i=0$, the proper transform $\Lambda_\z$ of $D_\z$ under the  map $\vphi\colon \M_{0,n+2}\ra \M_{1,n}$ identifying marked points $n+1$ and $n+2$ is an is an irreducible divisor. Furthermore, $\Lambda_\z$ is specified in the sense of \ref{basics} by the polynomial

\begin{equation}\label{polypoly} \frac{ \prodl_{a_i \geq 0} \!\!(x_{n+1}\! -\! x_i)^{|a_i|} \!\! \prodl_{a_i \leq 0}\!\!(x_{n+2}\!-\!x_i)^{|a_i|}\!\! -\!\! \prodl_{a_i \leq 0}\!\! (x_{n+1} \!-\! x_i)^{|a_i|}\!\! \prodl_{a_i \geq 0}\!\!(x_{n+2}\!-\!x_i)^{|a_i|}  }{x_{n+1}-x_{n+2} },\end{equation} and

$$\pi_{n+3}\I(\Lambda_\z) \sim dH- \sum m_I E_I^{n+3},$$ with coefficients as follows:

\begin{itemize}
\item If $n+1, n+2 \notin I$ and $\{ i \col a_i \neq 0\} \subset I$, then $m_I =0$.
\item If $n+1,n+2 \notin I$ and $\{i \col a_i \neq 0 \} \not\subset I$, then $m_I= (\sum_{i \notin I} |a_i|)-1.$
\item If $n+1, n+2 \in I$ and $\{ i \col a_i \neq 0 \} \subset \{1, \dots, n \} \sm I $, then $m_I =1$. 
\item If $n+1, n+2 \in I$ and $\{ i \col a_i \neq 0 \} \not\subset \{1, \dots, n \} \sm I $ then $m_I =0$. 
\item If $|\{n+1, n+2\} \cap I| =1$, then $m_I= \min \{ \sum_{0 \leq a_i \notin I}|a_i|, \sum_{0 \geq a_i \notin I}|a_i|\}. $
\item $d = (\sum_i |a_i|)-1.$
\end{itemize} \end{Theorem}The above theorem immediately yields a number of useful formulae, which we record prior to proving the theorem.
\begin{Corollary}\label{an0} If $\z = (a_1, \dots, a_n)$ with $a_i \neq 0$ for all $i$, then $\pi_{n+3}\I(\Lambda_\z) \sim dH - \sum_i m_I E_I^{n+3},$ with coefficients as follows: \begin{itemize}
\item If $n+1, n+2 \notin I$, then $m_I= (\sum_{i \notin I} |a_i|)-1.$
\item If $n+1, n+2 \in I$ then $m_I=0$ except when $I = \{n+1,n+2\}$, in which case $ m_{I}= 1$.
\item If $|\{n+1, n+2\} \cap I| =1$, then $m_I= \min \{ \sum_{0 \leq a_i \notin I}|a_i|, \sum_{0 \geq a_i \notin I}|a_i|\}. $
\item $d = (\sum_i |a_i|)-1.$
\end{itemize}
\end{Corollary}
\begin{proof} A special case of \ref{classeqs}. \end{proof}

\begin{Corollary}\label{classeqsred} Given $\z=(a_1, \dots, a_n)$ with $\sum_i a_i=0$, the class of $\Lambda_\z$ with respect to the Kapranov basis in index $r$ is
$$ ( \suml_{\substack{1 \leq i \leq n \\ i \neq r}} |a_i|-1 )H - \suml_{\substack{n+1,n+2 \notin I \\ \{i \col a_i \neq 0\} \not\subset I}}\bigg (\suml_{i \notin I \cup \{r\} } |a_i| -1\bigg)E_I$$ \begin{equation}\label{nastynasty5} -\suml_{| \{n+1,n+2 \} \cap I| =1}\min \Big{\{} \suml_{\substack{0 \leq a_i \\  i \notin I \cup \{r\} }}|a_i|, \suml_{\substack{0 \geq a_i \\ i \notin I \cup \{r\} }} |a_i|\Big{\}}  E_I - \sum_{\substack{ n+1,n+2 \in I \\ I\cap N \subset \{ i \col a_i =0\} \\ r \notin I}} E_I. \end{equation} 
\end{Corollary}

\begin{proof} Apply \ref{n+1ton} to \ref{classeqs}. Note that the last terms in \eqref{nastynasty5}, i.e. those involving $E_I$ for $\{n+1,n+2\} \in I$, vanish if $a_r \neq 0$. \end{proof}

\begin{Corollary}\label{nastynasty} Suppose that $\z=(a_1, \dots, a_n)$ with $a_i \neq 0$ for all $i$. Then the class of $\Lambda_{(a_1,\dots,a_n)} \subset \M_{0,n+2}$ with respect to the Kapranov basis in index $r$ for $r \in \{1, \dots n\}$ is

$$ ( \suml_{\substack{1 \leq i \leq n \\ i \neq r}} |a_i|-1 )H - \suml_{\substack{n+1,n+2 \notin I}} (\suml_{i \notin I \cup \{r\} } |a_i| -1)E_I$$ \begin{equation}\label{nastynasty1} -\suml_{| \{n+1,n+2 \} \cap I| =1}\min \Big{\{} \suml_{\substack{0 \leq a_i \\  i \notin I \cup \{r\} }}|a_i|, \suml_{\substack{0 \geq a_i \\ i \notin I \cup \{r\} }} |a_i|\Big{\}}  E_I. \end{equation} 
\end{Corollary}

\begin{proof} Apply \ref{n+1ton} to \ref{an0}. \end{proof}

We now prove the main result.
\begin{proof}[Proof of \eqref{classeqs}]
We can describe the interior of $\Lambda_\z$ as follows:
 $$[ \mb{P}^1;p_1,\dots,p_{n+2}] \in \Lambda_\z \cap M_{0,n+2}$$

$$\Llra \vphi ([\mb{P}^1; p_1,\dots,p_{n+2}] )=[C; q_1,\dots,q_n] \in D_\z$$ 

$$ \Llra \E g \in k(C) \colon \dv(g) = \sum_{i=1}^{n}a_i q_i\text{, for $g$ regular and invertible at the node } $$

\begin{equation}\label{nodal} \Llra \E h \in k(\mathbb{P}^1) \colon \dv(h) = \sum_{i=1}^n a_i p_i \text{ and } h(p_{n+1} ) = h(p_{n+2}) \end{equation}

Evidently  $h(x) = \prod_{i=1}^n(x-p_i)^{a_i}$ on an appropriate affine chart for a representative of $[\Pj^1;p_1,\dots,p_{n+2}]$. Thus, the condition that $$x=[\Pj; p_1,\dots,p_{n+2}] \in \Lambda_{\bf a}$$ is equivalent to requiring that any $(\Aff^1; q_1, \dots, q_{n+2})$ mapping to $x$ under the map $\chi$ from Fulton-MacPherson configuration space satisfies  $h(q_{n+1})= h(q_{n+2})$. This gives an equation $F$ specifying $\Lambda_\z \cap M_{0,n+2}$:

\begin{equation}\label{F} \prodl_{a_i \geq 0} (x_{n+1}\! -\! x_i)^{|a_i|} \prodl_{a_i \leq 0}(x_{n+2}\!-\!x_i)^{|a_i|} - \prodl_{a_i \leq 0} (x_{n+1} \!-\! x_i)^{|a_i|} \prodl_{a_i \geq 0}(x_{n+2}\!-\! x_i)^{|a_i|}.\end{equation}

Note that \ref{irredG} implies that $\Lambda_\z= \vphi_{*}^{-1}(D_\z)$ is irreducible. Since $F$ is the correct equation for $\Lambda_\z$ on $M_{0,n+2}$, only boundary terms of the form $x_i-x_j$ for $i \neq j$ can divide $F$. 

\begin{Claim}\label{dividingclaim} For $m \in \mathbb N$, $(x_i - x_j)^{m}$ divides $F$ if and only if $$ \{ i,j\}= \{n+1,n+2\}$$ and $m=1$. \end{Claim} We obtain \ref{dividingclaim} in the course of proving the formula for classes: the claim is equivalent to the assertion that the multiplicity of $F$ along $\Diag_{\{i,j\}}$ is zero unless $\{i,j\} = \{n+1,n+2\}$, in which case it is $1$. Given the claim, we recover the equation of the theorem.

\par With notation as from \ref{polynomialclasses}, recall that given $f \in k[x_1,\dots,x_{n+2}]$ such that $\phi_{*}^{-1}(V(f))=\bp(D)$, the class of the pull-back $\pi_{n+3}\I(D) \subset \M_{0,n+3}$ is $$ dH - \suml_{\substack{1 \leq |I| \leq n-1 \\ I \subset \{1, \dots, n+2\}}} m_I E_I,$$ where $d$ is the degree of $f$ and $m_I$ is the multiplicity of $f$ along the partial diagonal $\Diag_J= \{x_i = x_j \col i,j \in J \}$ for $J = \{1,\dots, n+3\} \sm I$. Hence we must compute the multiplicity of $F$ from \eqref{F} along partial diagonals $\operatorname{Diag}_J$ with $4 \leq |J| \leq n+2$. The multiplicity along a diagonal will be the multiplicity at a general point. To compute the multiplicity at an arbitrary point $b=(-b_1, \dots, -b_{n+2})$, we make the substitution $x_i \mapsto x_i+b_i$ and determine the degree of the initial term of the resulting equation as a polynomial in $x_i$. To get the multiplicity at a general point $b \in \Diag_J$, we set $b_i = t$ for $i \in J$, and then compute the minimum degree among nonzero monomials as a polynomial in $x_i$. \par

There are several cases to consider. Throughout, we define $k_J$ to be the multiplicity of $F$ along a partial diagonal $\Diag_J$ and let $N:= \{1, \dots, n\}$. To simplify notation, define $s(x)=n+1$ if $x \geq 0$, $s(x)=n+2$ if $x <0$. 

\begin{enumerate}
\item $n+1, n+2 \in J.$ 
\begin{itemize}
\item For $\alpha \in  \{n+1,  n+2 \}$ and $i \in J \cap N$ we substitute $(x_\alpha -x_i) \mapsto (x_\alpha -x_i)$.
\item For $i \notin J \cap N$ we substitute $(x_\alpha -x_i) \mapsto (x_\alpha -x_i -b_i+t)$.
\end{itemize} The initial term as a polynomial in $x_i$'s is then:

\begin{equation}\label{144} \prodl_{i \in J \cap N} (x_{s(a_i)}-x_i)^{|a_i|} \prodl_{i \in J^c \cap N}(t-b_i)^{|a_i|}$$ $$ -   \prodl_{i \in J \cap N } (x_{s(-a_i)}-x_i)^{|a_i|}\prod_{i \in J^c \cap N}(t-b_{i})^{|a_i|}.\end{equation}

If $|J| \geq 3$ and $J \cap N$ is not contained in the set $$A_0:=\{i \col a_i =0\}$$ the summed terms have distinct  prime factors, so \eqref{144} is nonzero. Hence

 $$k_J = \sum_{i \in J}|a_i|,$$
 for $J$ containing some $i\in N$ with $a_i \neq 0$.   \par
If $J\cap N \subset A_0$, then \eqref{144} is indeed zero, but the entire polynomial will be the same as that obtained via the requisite substitution for computation of the multiplicity of $F$ along $\Diag_{\{n+1,n+2\}}$. This particular substitution results in a coefficient of $x_{n+1}$ is given by 

$$ \sum_{a_i > 0} |a_i|(t-b_i)^{|a_i|-1}\prodl_{j \neq i}(t-b_j)^{|a_j|}$$ $$-\sum_{a_i < 0} |a_i|(t-b_i)^{|a_i|-1}\prodl_{j \neq i}(t-b_j)^{|a_j|},$$ which is nonzero since the summands have pairwise distinct prime factors. This shows that $k_J=1$ for $J \cap N \subset A_0$. \par

In particular, the multiplicity of our equation along $V(x_{n+1}-x_{n+2})$ is $1$. This proves part of \ref{dividingclaim}: $$\max\{ m \col (x_{n+1} -x_{n+2})^m \text{ divides } F \} =1.$$ 
\item $n+1, n+2 \notin J.$ 
\begin{itemize}
\item \par For $i \in J$ we substitute $(x_\alpha -x_i) \mapsto (x_\alpha -x_i -t+b_\alpha)$.
\item \par For $i \in N \sm J$ we substitute $(x_\alpha -x_i) \mapsto (x_\alpha -x_i -b_i+b_\alpha)$.
\end{itemize}The constant term of the resulting polynomial in $x_i$ is

\begin{equation}\label{*} \prod_{i \in J} (t-b_{s(a_i)})^{|a_i|} \prod_{i \notin J}(b_i-b_{s(a_i)})^{|a_i|}\end{equation} $$ -  \prod_{i \in J} (t-b_{s(-a_i)})^{|a_i|}\prod_{i \notin J}(b_i-b_{s(-a_i)})^{|a_i|} .$$ If there is an $i \in N\sm J$ with $a_i \neq 0$, then we necessarily have a monomial $(b_i -b_{s(a_i)})$ dividing one term but not in the other, and the difference in nonzero. Hence $N \sm J \not \subset A_0$ implies that $k_J = 0$. Now suppose $a_i=0$ for all $i \in N \sm J$. In this case, \eqref{*} is zero. Let $r = \suml_{a_i > 0} |a_i| = \suml_{a_i < 0} |a_i|$. The next lowest term as a polynomial in $x_i$ includes a summand:
\begin{equation} \label{**}[ \, (t-b_{n+1})^{r-1}(t-b_{n+2})^{r} - (t-b_{n+1})^r(t-b_{n+2})^{r-1} \,] (\sum_{\substack{i \leq n \\ a_i > 0}} x_i)  \end{equation} There are other degree $1$ contributions, but these do not involve $x_i$ for $i \leq n$ and $a_i>0$, so to conclude $k_{ \{1,\dots,n \} } =1$ it suffices to note that \eqref{**} is nonzero. This shows $$k_J =  1 \text{  if }N \sm J  \subset \{ i \col a_i =0 \}$$
\begin{equation}\label{ijzero} k_J  = 0 \text{ otherwise. }\end{equation}

In particular, if all $a_i$'s are nonzero, then $k_J \neq 0$ if and only if $J = \{1, \dots, n\}$ in which case $k_J=1$.

\item$|\{n+1, n+2\} \cap J| =1$. Without loss of generality, assume $n+1 \in J$ and $n+2 \notin J$; the argument is symmetric. 
\begin{itemize}
\item For $i \in N \cap J $ we substitute $(x_{n+1} -x_i) \mapsto (x_{n+1} -x_i)$ and $(x_{n+2}-x_i) \mapsto (x_{n+2}-x_i -t+b_{n+2})$.
\item For $i \in N\sm J $ we substitute $(x_{n+1} -x_i) \mapsto (x_{n+1} -x_i-b_i+t)$ and $(x_{n+2}-x_i) \mapsto (x_{n+2}-x_i -b_i+b_{n+2})$.
\end{itemize} Define $h_{j,i}=(x_{n+j}-x_i)$ for $j \in \{1,2\}$ and $1 \leq i \leq n$. With this notation, substituting gives $F=$  

$$ \prod_{\substack{i \in J \\ a_i > 0}} \!\! (h_{1,i})^{|a_i|}\!\!  \prod_{\substack{i \notin J\\ a_i > 0}}\!\! (h_{1,i} -b_i+t)^{|a_i|}\!\! \prod_{\substack{i \in J \\ a_i < 0}}\!\! (h_{2,i}-t+b_{n+2})^{|a_i|}\!\! \prod_{\substack{i \notin J\\ a_i < 0}}\!\! (h_{2,i} -b_i+b_{n+2})^{|a_i|}\!\!-\!\!$$ $$ 
 \prod_{\substack{i \in J \\ a_i < 0}} \!\!(h_{1,i})^{|a_i|}\!\! \prod_{\substack{i \notin J\\ a_i < 0}} \!\!(h_{1,i}-b_i+t)^{|a_i|}\!\!  \prod_{\substack{i \in J \\ a_i > 0}}\!\! (h_{2,i}-t+b_{n+2})^{|a_i|} \!\! \prod_{\substack{i \notin J\\ a_i > 0}}\!\! (h_{2,i}-b_i+b_{n+2})^{|a_i|}\!\!.$$ The initial term of the expanded expression is:

\begin{equation}\label{145}  \prod_{\substack{i \in J \\ a_i > 0}} \!\!(x_{n+1} -x_i)^{|a_i|} \!\!  \prod_{\substack{i \notin J\\ a_i > 0}}\!\!(t-b_i)^{|a_i|}\!\!\prod_{\substack{i \in J \\ a_i < 0}}\!\! (b_{n+2}-t)^{|a_i|} \!\! \prod_{\substack{i \notin J\\ a_i < 0}}\!\! (b_{n+2}-b_i)^{|a_i|}- $$ $$
\prod_{\substack{i \in J \\ a_i < 0}}\!\! (x_{n+1} -x_i)^{|a_i|}\!\!  \prod_{\substack{i \notin J\\ a_i < 0}} \!\! (t-b_i)^{|a_i|} \!\! \prod_{\substack{i \in J \\ a_i > 0}}\!\!(b_{n+2}-t)^{|a_i|} \!\!\prod_{\substack{i \notin J\\ a_i > 0}}\!\! (b_{n+2}-b_i)^{|a_i|}.\end{equation}The two terms comprising \eqref{145} necessarily have distinct factors regardless of the relationship between $A_0$ and $J$, so that the difference is nonzero. Hence

$$k_J  = \min \{ \! \sum_{\substack{a_i \geq 0 \\ i \in J}}|a_i|, \sum_{\substack{a_i \leq 0 \\ i \in J}}|a_i| \}$$
\end{enumerate}

These formulae do not quite give the class of the divisor $\Lambda_\z$. Assuming \eqref{dividingclaim}, the actual equation specifying $\Lambda_\z$ is $\frac{F}{x_{n+1}-x_{n+2}}$. Hence the relevant multiplicities giving class coefficients are computed by subtracting the multiplicity of $(x_{n+1}-x_{n+2})$ along $\Diag_J$ from each $k_J$ computed above. Our formulae will therefore be as follows:

\begin{enumerate}
\item If $n+1, n+2 \in J$, substituting to compute the multiplicity along $\Diag_J$ gives $(x_{n+1}-x_{n+2}) \mapsto (x_{n+1} - x_{n+2})$ so the multiplicity is 1. Hence, defining $M= \{1, \dots, n+2 \}$:
$$m_{M\sm J}=k_J -1= (\sum_{i \in J} |a_i|)-1,$$ for $J\cap N \not\subset \{ i \col a_i =0\}$, and

$$m_{M \sm J} = 0,$$ for $J \cap N \subset \{ i \col a_i=0\}$. 
\item If $n+1, n+2 \notin J$, we substitute $(x_{n+1}-x_{n+2})\mapsto(x_{n+1}-x_{n+2}+b_{n+1}-b_{n+2})$, which shows the multiplicity of $x_{n+1} - x_{n+2}$ is zero along $\Diag_J$. Hence $$m_{M \sm J} = 0$$ unless $M \sm J \subset A_0$, in which case
$$m_{M \sm J} =k_J= 1$$
\item $|\{n+1, n+2\} \cap J| =1$: Evidently the multiplicity of $(x_{n+1}-x_{n+2})$ here is also zero and
$$m_{N \sm J}=k_J= \min \Big{\{} \sum_{\substack{i \in J \\ a_i \geq 0}}|a_i|, \sum_{\substack{i \in J \\ a_i \leq 0}}|a_i|\Big{\}}. $$ \end{enumerate} Reformulating (1)-(3) above gives the theorem. \par 

It remains to complete the proof of \ref{dividingclaim}.  We have already noted that $(x_{n+1}-x_{n+2})^k$ divides $F$ if and only if $k=1$. To see that no other $(x_i -x_j)$ divides $F$ for $i\neq j$ and $i,j\leq n$, recall that  $F$ has multiplicity zero along each  partial diagonal $V(x_i -x_j)$ by \eqref{ijzero}. From inspection of \eqref{F}, it is evident that neither $(x_{n+1}-x_j)$ nor $(x_{n+2}-x_j)$ can divide $F$ for $j \leq n$.  \end{proof}

\begin{Example}\label{exx} Let $D_k:= \Lambda_{(k,1,-1,-1,\dots,-1)} \subset \M_{0,k+5}$. Let $K=\{k+4,k+5\}.$ We apply the formulas from \ref{nastynasty} to compute the class of $D_k$ with respect to the Kapranov basis for $\M_{0,k+5}$ using index $1$.  Note that in our case 
$$ \suml_{\substack{1 \leq i \leq k+3 \\ i \neq 1}} |a_i|-1 = \big{(}\sum_{2 \leq i \leq k+3}( 1)\big{)} -1 = k+1,$$ and for $K \cap I = \0$,
$$\suml_{\substack{i \notin I \cup \{1\}}} |a_i|-1=  |\{2, \dots , k+3 \} - I | -1 = k+2 - |I|-1 = k+1 - |I|.$$

For $| K \cap I| =1$, the  coefficient is

$$\min \Big{\{} \suml_{\substack{0 \leq a_i \\  i \notin I \cup \{1\}}}|a_i|, \suml_{\substack{0 \geq a_i \\ i \notin I \cup \{1\} }}|a_i|\Big{\}}.$$  If $2 \in I$, then the minimum is zero; if not, the minimum is always one, since cardinality considerations show $$ \{i \col a_i \leq 0\} \not\subset I.$$ Hence we have that
\begin{equation}\label{Dk} D_k \sim (k+1)H - \suml_{i=1}^{k}\Big{ (}\suml_{\substack{K\cap I=\0 \\ |I|=i}} (k+1-i) E_I \Big{)} - \suml_{\substack{|K\cap I|=1 \\ 2 \notin I}} E_I.   \end{equation}
\end{Example}

\begin{Example}\label{LM} Consider the  $(k+1)$-tuple $(k,-1,-1,-1, \dots, -1)$ which gives a divisor on $\M_{0,k+3}$. Note that, by \ref{restricting}, this divisor can be obtained by intersecting $D_k$ from \ref{exx} with the boundary where the first two markings ``collide." Using  \ref{nastynasty}, we compute the class of $L_k :=\Lambda_{(k,-1,\dots,-1)}$ with respect to the index $1$ Kapranov basis:

$$L_k \sim (k-1) H - \sum_{i=1}^{k-1} \sum_{\substack{ |I|=i \\ I \subseteq \{2,\dots, k+1\} }} \!\! \big(k-1-i\big) E_I.$$ Note that all $E_I$ with $k+2 \in I$ or $k+3 \in I$ do not contribute to the class of $L_k$. 
\end{Example}

We return to general results on Chen--Coskun divisors. The next theorem relates certain Chen--Coskun divisors to spherical hypertree divisors (defined in \ref{ht}).

\begin{Theorem}\label{sphere} If ${ \bf a}=(1,1,\dots, -1, -1, \dots)$ is a $2k$-tuple with $\sum a_i =0$, then $ \Lambda_{\z} = D_{\Gamma}$ where $\Gamma$ is the spherical hypertree divisor associated to a bipyramidal bicolored spherical triangulation with $2k$ triangles.  \end{Theorem} 
\begin{proof} Since $D_\Gamma$ and $\Lambda_\z$ are irreducible, it suffices to show $D_\Gamma \cap M_{0,n} \subset \Lambda_\z$. For this, we appeal to a characterization of spherical bipyramid hypertree divisors given in \cite[ 9.5]{CT}. 

Let $\Gamma$ be the spherical bipyramid divisor on $n=2k+2$ vertices. Then there is a partition of $1,\dots,n$ into subsets $X,Y,Z$ with $|X|=|Y|=k$, and $|Z|=2$ where the indices in $Z$ correspond to ``poles" of the bipyramid and those in $X$ and $Y$ are alternating points on the ``equator:"

\begin{figure}[h]
\center
\includegraphics[scale=.3]{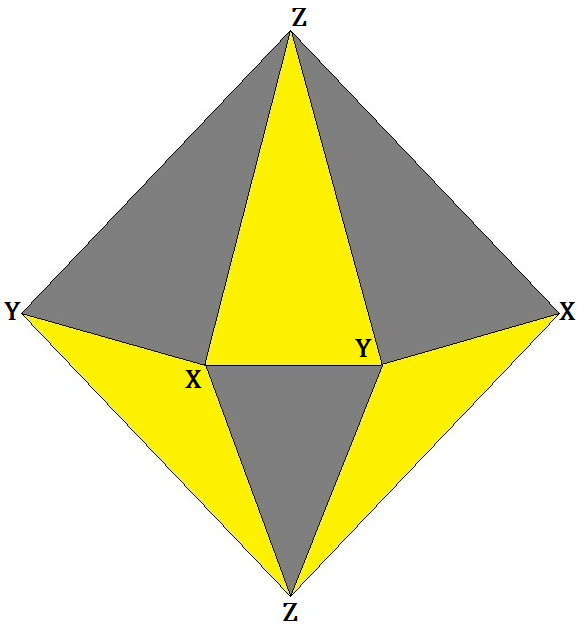}
\caption{ Spherical pyramidal triangulation with subsets $X$, $Y$ and $Z$ indicated.}
\end{figure}

Assume $X = \{1,\dots,k\}$, $Y=\{k+1,\dots,2k\}$, $Z=\{2k+1,2k+2\}$.  Consider the embedding $\eta$ of $ [\mb{P}^1; p_1,\dots,p_n]$ into $\mb{P}^{k}$ as a rational normal curve degree $k$; let $q_i= \eta(p_i)$, and $$L=\langle q_i \rangle_{i \in Z},$$  $$ \tilde{X}=\langle q_i \rangle_{i \in X}, $$ $$\tilde{Y}= \langle q_i \rangle_{i \in Y} .$$ Castravet and Tevelev show $D_\Gamma$ consists of $[\mb{P}^1;p_1,\dots,p_n]$ such that 

\begin{equation} \label{&} L \cap \tilde{X} \cap \tilde{Y} \neq \0. \end{equation}  

 Fix a representative of marked points $p_i$ for an arbitrary element in $D_\Gamma \cap M_{0,n}$. The function with zeros order one at $p_i$ for $i \in X$ and poles order one at $p_i$ for $i \in Y$ is given by $h=\eta^*(h_1/h_2)$ where $h_1$ is a linear equation of $\tilde{X}$ and $h_2$ one of $\tilde{Y}$. Let $q_{n+i} = [a_{1i}: \dots : a_{ki}]$ for $i = 1,2$. Then \eqref{&}  implies that, for some elements $s$ and $t$ in the base field,

$$h_j ([sa_{11}+ta_{12}: \dots : sa_{k1}+ta_{k2}])=0,$$ for $j=1$ or $j=2$. This together with linearity implies that 

$$\frac{h_1}{h_2}(q_{n+1})=\frac{h_1}{h_2}(q_{n+2}),$$ so that

$$\eta^*\frac{h_1}{h_2}(x_{n+1})=\eta^*\frac{h_1}{h_2}(q_{n+2}).$$

Hence $\dv(h)=p_1+\dots+p_k-p_{k+1}-\dots -p_{2k}$ and $h(p_{2k+1})=h(p_{2k+2})$, so $h$ witnesses that $[\Pj;p_1,\dots,p_n] \in \Lambda_{(1,\dots,1,-1,\dots,-1)}$.  \end{proof} The next theorem describes how any Chen--Coskun divisor arises from intersections of a ``universal divisor" of the form $\Lambda_{(1,1,\dots,-1,-1,\dots)}$ with boundary divisors. Together with the previous result, this gives a relationship between Chen--Coskun divisors and hypertree divisors: all  Chen--Coskun divisors are obtained by a sequence of restrictions of a bipyramidal spherical hypertree divisor.

\par Note that, in an attempt to clarify the proof of the theorem, we use labels $0, \dots, n+2$ for markings on $\M_{0,n+3}$ and markings $1,\dots, n+2$ on $\M_{0,n+2}$.

\begin{Theorem}\label{restricting}  Let $\z=(a_0,\dots,a_{n}) \in \Z^{n+1}$ be such that $\sum_i a_i =0$; and $\gcd(a_0,\dots,a_{n})=\gcd(a_0+a_1,a_2,\dots,a_{n-1},a_n) =1$. Define ${\bf b} = (a_0+a_1,\dots,a_{n-1},a_n) \in \Z^n$. Then  $\Lambda_\z \cap \delta_{\{0,1\}} = \Lambda_{{\bf b}}$ as a divisor on $\delta_{\{0,1\}} \simeq \ol{M}_{0,n+2}$.  \end{Theorem}
\begin{proof} 
Consider the following diagram:
\begin{equation}\label{diag2}\begin{CD}
\mb{A}^{n+3}  @<\phi<< \Aff^1[n+3]  @.@>\chi>> \ol{M}_{0,n+3}    \\
 \bigcup \! |   @.  \bigcup \! |  @ . @. \bigcup \! |  \\            
\Diag_{\{0,1\}}  \!\! @<\bar\phi<< \Delta_{\{0,1\}} @.@>\bar\chi>> \delta_{ \{0,1\}}   \\ 
@ A \nu AA                 @V p_0 VV \!\!\!\!\!\!\!\!\!\!\!\!\!\!\!\!\!\! @AA\tau_{0} A  @ V \pi_{0} VV \!\!\!\!\!\!\!\!\! @AA\sigma_{0}A     \\
 \Aff^{n+2} @<\phi<<  \Aff^{1}[n+2]     @.   @>\chi>> \ol{M}_{0,n+2}      
\end{CD}\end{equation} \\

The maps $\phi$ and $\chi$ are as in \eqref{diag1}. The maps $\balpha$ and $\bbeta$ are restrictions of these to the indicated subsets. The isomorphisms of $\delta_{ \{0,1\} }$ with $\M_{0,n+2}$ and $\Delta_{\{0,1\}}$ with $\Aff^{1}[n+2]$ are restrictions of the index $0$ forgetful morphisms; the inverse maps are given by index $0$ sections, $\sigma_{0}$ and $\tau_{0}$. The isomorphism $\nu$ is defined by $$(x_1, \dots, x_n,x_{n+1},x_{n+2}) \mapsto (x_1,x_1, \dots, x_n,x_{n+1},x_{n+2}).$$ (That is, the map ``repeats the first index.")

\par Commutativity of the lower left rectangle follows from the fact that the iterated blow-up defining $\Aff^1[n+3]$ restricts to an iterated blow-up of the subspace $\Diag_{\{0,1\}}$; this coincides with the Fulton-MacPherson construction when $\Diag_{\{0,1\}}$ is naturally identified with $\Aff^{n+2}$. Commutativity of the lower right rectangle is immediate from that of \eqref{diag1}.

 Define\\
$$\partial =  \bigcup_{\substack{|I| \geq 2 \\ I \neq \{0,1\} }}\!\!\! \delta_I$$
$$V = \bigcup_{ \{i,j\} \neq \{0,1\}}\!\!\!\! \Diag_{\{i,j\}},$$
 $$\delta_{ \{0,1\} }^0=\delta_{ \{0,1\} } \sm \partial,$$  
$$A= \Aff^{n+3} \sm V,$$
$$M = M_{0,n+3} \cup \delta_{ \{0,1\} }^0,$$
$$B = \Aff^{n+2} \sm \{\text{ diagonals }\}.$$\\  Let $\Lambda_{\z}^0= \Lambda_{\z} \cap M$ and $\Lambda_{{\bf b}}^0 = \Lambda_{{\bf b}} \cap M_{0,n+2}$. We will show that $\Lao \cap \delta_{ \{0,1\} }= \Lbo$, so $\Lambda_{\z} \cap \delta_{ \{0,1\} }$ and $\Lambda_{ {\bf b}}$ can differ only by boundary divisors of $ \delta_{ \{0,1\} }\simeq \M_{0,n+2}$. Given this, for equality of the divisors it will suffice to show that they have the same classes. 

Let $F$ be the polynomial in $k[x_0, \dots, x_{n+2}]$ specifying $\Lambda_\z$; the form of $F$ is given in \eqref{polypoly}. Since $\Lao \cap \delta_{ \{0,1\} }= \sigma_0^{-1}(\Lambda_\z \cap M)$ and the diagram commutes, we have that

$$\bp(\Lao \cap \delta_{ \{0,1\} })= \chi^{-1}(\sigma_0^{-1}(\Lambda_\z \cap M)) = \tau_0^{-1}(\bp(\Lambda_\z \cap M)) $$ $$= \tau_0^{-1}(\phi_*\I(V(F) \cap A)) 
= \phi_*\I(\nu^{-1}(V(F)) \cap B)= \phi\I(V(\nu^{*}F) \cap B).$$ \\
But $h^{*}F$ is simply $F(x_0, \dots, x_n,x_{n+1},x_{n+2})$ with $x_1$ substituted for $x_0$. If both $a_1$ and $a_0$ are non-negative, we quite literally add exponents and obtain the equation $G(x_1,\dots, x_{n+2})$ which specifies $\Lambda_\zb$ on $\M_{0,n+2}$. If $a_0 \geq 0$ and $a_1<0$, the substitution yields

$$h^{*}F = \big((x_{n+1} - x_1)(x_{n+2}-x_1)\big)^{\min\{ |a_0|,|a_1|\}}G.$$ Since the factors $(x_{n+1} - x_1)$ and $(x_{n+2}-x_1)$ contribute boundaries, the divisor specified by $h^*F$ coincides with that specified by $G$ on $M_{0,n+2}$. Hence $\Lambda_{\z} \cap \delta_{ \{0,1\} }$ and $\Lambda_\zb$ are specified by the same polynomial on the interior of $\M_{0,n+2}$, as claimed.

We now show that the classes of $\Lambda_{\z} \cap \delta_{ \{0,1\} }$ and $\Lambda_{ {\bf b}}$ are the same. Note that $H$ and $\delta_I$ for $|I \cap \{0,1\}|=1$ restrict to the zero class on $\delta_{\{0,1\}}$. For $\{0,1\} \subsetneq I$, we have that $\delta_I \cap \Lambda_\z = \delta_{ I \sm \{0\}}$ on $\delta_{\{0,1\}}$ naturally identified with $\M_{0,n+2}$ with markings $1, \dots, n+2.$ The divisor $\delta_{\{0,1\}}$ restricts to $-H$ on $\delta_{\{0,1\}}$, the negative of a Kapranov hyperplane class with respect to the index $1$ Kapranov basis. Applying this to the class of $\Lambda_\z$, we see that if

$$\Lambda_\z \sim dH - \sum m_I E_I,$$ then 

$$\Lambda_\z \cap \delta_{\{1\}} \sim m_{\{1\}} H - \sum_{ \{1\} \subsetneq I} m_I \delta_{I}.$$ 

Note that $1 \leq |I| \leq n-1$ so that $1 \leq |I \sm \{1\}| \leq n-2$ for $\{1 \} \subsetneq I$, and $\delta_I = E_{I \sm \{1\}}$ on $\M_{0,n+2}$ using the Kapranov basis in index $1$. Hence, to show that $\Lambda_\z \cap \delta_{\{0,1\}}$ has the same class as $\Lambda_\zb$, we must verify that if $\Lambda_\zb = gH - \sum n_I E_I$, then $m_{\{1\}} = g$ and $n_{I \sm \{1\}}=m_I$. \par

By \ref{classeqsred}, we see that $g = \sum_{i \neq 1} |b_i|-1= \sum_{i \neq 0,1} |a_i|-1=m_{\{0,1\}}$ as desired. Moreover, $$|\{n+1,n+2\} \cap I |= | \{ n+1, n+2 \} \cap I \sm \{1\}|,$$ $$\{i \col a_i \neq 0 \} \subset I \iff \{ i \col b_i \neq 0 \} \subset I \sm \{1\},$$ and $$\{i \col b_i \neq 0\} \subset \{2,\dots, n\} \sm I \iff \{ i \col a_i \neq 0\} \subset \{1,\dots,n \} \sm I.$$ This implies that the coefficient of $E_I$ in the class of $\Lambda_\z$ and of $E_{I \sm \{1\}}$ in the class of $\Lambda_\zb$ are computed using the same formula from \eqref{classeqsred}. Noting that the formulas depend only on sums over the complements of $I \cup \{0\}$ and $I$, respectively, so that $m_I = n_{I \sm \{1\}}$ as desired. \end{proof}

%---------------------------------------------------------------CoUNTEREXAMPLE----------------------------------------------------------
\section{Counterexample to the Castravet--Tevelev conjecture}\label{counterex}
In \ref{exx}, we computed the class of $D_k=\Lambda_{(k,1,-1,\dots)}$ on $\M_{0,k+5}$ with respect to the index $1$ Kapranov basis: 

\begin{equation}\label{specialclass}D_k \sim (k+1)H - \!\!\!\suml_{\substack{K\cap I=\0 \\ |I|=1}}\!\!\!  k E_I -\!\!\! \suml_{\substack{K\cap I=\0 \\ |I|=2}}\!\!\! (k-1) E_I- \dots -\!\!\!  \suml_{\substack{K\cap I=\0 \\ |I|=k}}\!\!\! E_I - \suml_{\substack{|K\cap I|=1 \\ 2 \notin I}} E_I,\end{equation}   where we define $K=\{k+4,k+5\}$.

The divisor $D_k$ is evidently effective. For extremality we appeal to the criterion given by \ref{edgefact}: we construct an irreducible covering family of curves with $C \cdot D_k < 0$. Define a Kapranov map $\psi_1$ from $\ol{M}_{0,k+5}$ to $\mb{P}^{k+2}$. Let $p_2,\dots,p_{k+5}$ be the points in $\mb{P}^{k+2}$ such that $E_I \mapsto \langle p_i \rangle_{i \in I}$. Inspection of \eqref{specialclass} shows that the image  $S=\psi_1(D_k)$ is a hypersurface of degree $k+1$ with a point of multiplicity $k$ at each $p_i$ for $2 \leq i \leq k+3$. Moreover, we have $2k+2$ codimension $2$ subspaces $\langle p_i \rangle_{i \in J}$ for $|J|=k+1$, $|K \cap J|=1$ and $2 \notin J$ which are contained with multiplicity $1$. To see this last fact, note that there are $k+1$ subsets of $\{3, \dots, k+3\}$ of size $k$, obtained by omitting a single index. Augmenting these subsets with either index $k+4$ or index $k+5$ gives $2k+2$ codimension $2$ spans as claimed. 

In $\Pj^{k+2}$, consider the family of curves $\mathcal{G}$ obtained by intersecting a 2-plane through $p_2$ with $S$. Let $\mathcal{F}$ denote the covering family of $D_k$ obtained by taking proper transforms of curves in $\mathcal{G}$ with respect to $\psi_1$.

\begin{figure}[h]\label{cubicthreefoldpic}
\center
\includegraphics[scale=0.35]{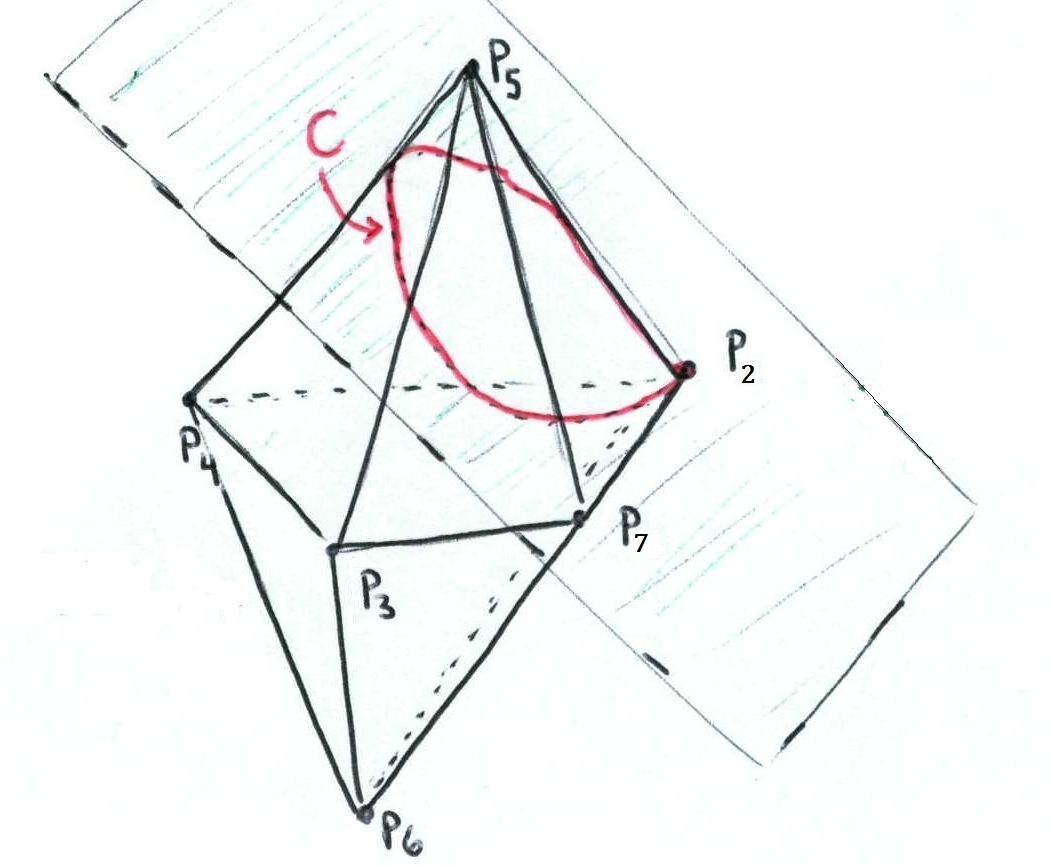}
\caption{Constructing a covering family for the image of $D_2$ in $\mb{P}^4$ under the Kapranov morphism in index $1$.}
\end{figure}

\begin{Lemma}\label{pairing} A general curve $C$ in the covering family $\mathcal{F}$ of $D_k$ has intersection pairing $-1$ with $D_k$. \end{Lemma}
\begin{proof} By construction, the image $\psi_1(C)$ in $\mb{P}^{k+2}$ will in intersect a hyperplane in $k+1$ points; passes through $p_2$ with multiplicity $k$; and transversally intersects the $2k+2$ codimension 2 linear spans $\langle p_i \rangle_{i \in I}$ for $|I|=n+1$, $2 \notin I$, and $|I \cap \{k+4,k+5\}|=1$ that contribute to the class of $D_k$.  Hence $C \cdot D_k = (k+1)(k+1) -(k)(k) - (1)(2k+2) = k^2+2k+1-k^2-2k-2=-1.$\end{proof} 

\begin{Lemma} A general curve $C$ in the covering family $\mathcal{F}$ of $D_k$ is irreducible. 
\end{Lemma}
\begin{proof} Note that it will suffice to prove that $\psi_1(C)$ is irreducible, i.e. that a general curve in $\mathcal{G}$ is irreducible. Let $T$ be the union of all lines through $p_2$ that are contained in $S$. Note that $T$ must have codimension at least $2$, since otherwise $S$ contains a codimension $1$ cone over $p_2$ and by irreducibility $S$ itself is a cone over $p_2$. However, $p_2$ is a point of multiplicity one less than the degree of $S$, so this is a contradiction. \par

 If we consider the map $\pi\colon \Pj^{n+2} \sm \{p_2\} \ra \Pj^{n+1}$ that projects from the point $p_2$, the image of $T$ is a subvariety of codimension at least $2$ and the image of a $2$-plane through $p_2$ is a line. Hence, for a general $2$-plane $h$ containing $p_2$, $\pi(h) \cap \pi(T) = \0$. We can reformulate this statement as follows: for a general $2$-plane $h \subset \Pj^{k+2}$ containing $p_2$, the curve $S \cap h$ contains no line through $p_2$. 

Now, for a contradiction, suppose that the intersection of a general $2$-plane with $S$ is reducible. Then the plane curve obtained via intersection is the union of a curve $g_1$ of degree $m_1$ and a curve $g_2$ of degree $m_2$, for $m_1,m_2 \geq 1$. Without loss of generality, $p_2$ is a point of multiplicity $m_1$ on $g_1$. But then $g_1$ is a union of lines through $p_2$. \end{proof} It is shown in the next section that the preceding two lemmas imply

\begin{Corollary}\label{corcor} For each $k$, $D_k$ generates an extremal ray of the effective cone of $\M_{0,k+5}$. \end{Corollary}

We now verify that $D_k$ is not linearly equivalent to a hypertree divisor or hypertree divisor pull-back for $k \geq 2$. To this end, consider the class of $\pi_{k+6}\I(D_k) \subset \M_{0,k+6}$ with respect to the index $k+6$ Kapranov basis. By \ref{classeqs}, 
$$ \pi_{k+6}\I(D_k) \sim (2k+1)H - \dots.$$ From the proof of \ref{htpoly}, a hypertree divisor on less than or equal to $t$ vertices is specified by a polynomial of degree at most $t-3$. Hence, given a hypertree divisor or hypertree divisor pull-back $D_\Gamma \subset \M_{0,k+5}$, we have that $$\pi_{k+6}\I(D_\Gamma) \sim sH - \dots$$ where $s \leq k+2$. Evidently $s=2k+1$ is impossible unless $k=1$.

%----------------- Conditions -------------------------------

\section{Covering families of curves and conditions for extremality}\label{conditions}
Results of this section will imply \ref{corcor}. In fact, a sufficient result for \ref{corcor} is proved in \cite[4.1]{CC}: Chen and Coskun show that if $D$ is an irreducible divisor and there exists an irreducible curve $C$ so that $C \cdot D <0$ and $D$ is covered by irreducible curves numerically equivalent to $C$, then $D$ generates an extremal ray of the pseudoeffective cone. \par
We prove a slightly stronger result: under the same hypotheses, $D$ generates an ``edge" of the pseudoeffective cone. Roughly, this means that $D$ is extremal and additionally the boundary of the pseudoeffective cone is not rounded near the ray generated by $D$. The proof, completed in \ref{edgefact}, follows from two lemmas of convex geometry. We first set up some notation; throughout we use standard Euclidean notions of distance, boundedness, etc. on $\R^N$ with the usual coordinates.\par

Given a convex cone $X \subset \R^N$, we say that $v \in X$ is an {\bf edge for $X$} if there exist linear functions $h_1, \dots, h_{N-1}$ so that \begin{equation}\label{h1} \bigcap_{i=1}^{N-1} \{ h_i = 0\} = \langle v \rangle \end{equation} and \begin{equation}\label{h2}X \subset \bigcap_{i=1}^{N-1} \{ h_i \geq 0\}.\end{equation} We say that $v \in X$ is {\bf extremal} if $v = a_1 w_1+a_2w_2$ for $a_1,a_2 \geq 0$ and $w_1,w_2 \in X$ implies that $w_1$ and $w_2$ are proportional to $v$. \par

\begin{Lemma}\label{L1} If $X$ is a convex cone in $\R^N$ and $v \in X$ is an edge, then $v$ is extremal. \end{Lemma}

\begin{proof} If $v = a_1 w_1 + a_2 w_2$, then $0=h_i(x)=a_1 h_i(w_1) + a_2 h_i(w_2)$ for each $1 \leq i \leq N-1$. Since $w_i \in X$, $h_i (w_i) \geq 0$ for each $i$ by \eqref{h2}, and since $a_j \geq 0$ we must have $h_i(w_1) = h_i(w_2)=0$ for all $i$. Hence by \eqref{h1}, $w_j \in \langle v\rangle$ as desired.   \end{proof}

Given a collection of points $V \subset \R^N$, let $\C(V)$ denote the closure of the convex hull of all non-negative multiples of elements in $V$. This is, in particular, a closed convex cone in $\R^N$.

\begin{Lemma}\label{L2} Given $V \subset \R^N$and $v \in V$, suppose that \begin{itemize}
\item[(a)] There exists a linear function $\sigma$ so that $\sigma(x) <0$ for $x \in V$ if and only if $x$ is a positive multiple of $v$; 
\item[(b)] There exists an affine hyperplane $0\notin H \subset \R^N$ with $H \cap  \C(V)$ nonempty and bounded. \end{itemize}
Then $v$ is an edge of $\C(V)$. Moreover, $v$ is extremal.  \end{Lemma}

\begin{proof} 

We assume throughout that $V$ contains at least two points that are not multiples of each other, since the lemma is clear when $V$ is a ray or line. \par

For any $x \in V$, some positive multiple of $x$ lies in $H$. To see this, note that without loss of generality $H = \{ z \col g(z) = 1 \}$ for some linear function $g$. If $\lambda x \notin H$ for all $\lambda \geq 0$, we must have $g(x) =b \leq 0$. Take any $y \in H \cap \C(V)$ not a multiple of $x$. Such a $y$ exists by the assumption that $V$ contains linearly independent points. If $b=0$, then $y + \lambda x \in C(V) \cap H $ for all $\lambda >0$, which contradicts boundedness of $H \cap \C(V)$. \par

 If $b<0$, let $\lambda_1 > 1$ and let $\lambda_2= \frac{\lambda_1 -1}{|b|}>0$. Then $ \lambda_1 y + \lambda_2 x \in H \cap \C(V)$, since $\lambda_1, \lambda_2>0$ and $g(\lambda_1 y + \lambda_2 x) = 1$. Since $x$ and $y$ are linearly independent and $\lambda_2 \ra \infty$ as $\lambda_1 \ra \infty$, this gives an unbounded sequence in $H \cap \C(V)$, again contradicting (b). Let $K:= \C(V) \cap H$. We have shown that $K$ is closed, bounded, and convex, and that $$ \C(V) = \{ \lambda x \col x \in K, \lambda \geq 0 \}.$$   These observations will be used later. \par

Now let $T$ denote the subspace $\{ y \col \sigma(y) =0 \} \subset \R^N$, where $\sigma$ is supplied by (a).
 \begin{Claim}\label{c1} There exists a basis $\{v,x_1, \dots, x_{N-1}\}$ for  $\R^N$ so that, if we let $h_i$ denote the coordinate function naturally associated to the element $x_i$ of the basis $\{v,x_i\}$:

\begin{itemize} \item $x_i \in T$ for $1 \leq i \leq N-1$. \item $K \cap T \subset \bigcap_{i=1}^{N-1} \{ h_i \geq 0 \}.$  \end{itemize} \end{Claim} 

Given the claim, since $\C(V) \cap T$ consists of non-negative multiples of elements of $K\cap T$, it follows that $\C(V) \cap T$ also lies in this intersection of half-spaces. With notation as in the claim, we have that $$V \subset \bigcap_{i=1}^{N-1} \{ h_i \geq 0 \}.$$ Indeed, let $y \in V \sm \{v\} $. Then $$y = -\alpha v + \sum_i a_i x_i$$ for some uniquely determined coefficients $\alpha, a_i$. Moreover, $$\sigma(y) =-\alpha  \sigma( v) \geq 0$$ by assumption (a). Since $\sigma (v)<0$, we must have $\alpha \geq 0$. Then $\alpha v + y = \sum_i a_i x_i \in \C(V) \cap T$, so that $a_i \geq 0$ by choice of the basis $x_i$.  \par This shows that $$\C(V) \subset \bigcap_{i=1}^{N-1} \{ h_i \geq 0\}$$ and $$\bigcap_{i=1}^{N-1} \{ h_i =0 \} = \langle v \rangle,$$ so that $v$ is an edge for $\C(V)$ as desired. \par

To prove the claim \ref{c1}, let $K_0:=K \cap T$. Since $K=\C(V) \cap H$, we have that $$K_0 \subset H \cap T = \{ x \in T \col g(x) = 1 \},$$ where as before $g$ is a linear function on $\R^N$ so that $H = \{ x \in \R^N \col g(x) = 1\}$. This subset is nonempty, since $\sigma(v)<0$ but $\sigma(x)\geq 0$ for some $x \in V$, so that we can find $x' \in C(V)$ with  $\sigma(x')=0$. As argued previously, some positive multiple of $x'$ lies in $H$, therefore in $K_0$. \par

 If we take $T_0$ to be the $(N-2)$-dimensional subspace of $T$ where $g$ vanishes, we have $T_0 \cap K_0 = \0$. Let $x_1'$ be a normal vector to $T_0$ with $g(x)>0$. Let $x_2', \dots, x_{N-1}'$ be a basis for $T_0$. Then let $h_i'$ denote the coordinate functions associated to this basis, and by boundedness of $K_0$ we have that for each $i$, $h_i'(K_0) \subset [a_i, b_i]$ for some finite $a_i,b_i$. Note that by assumption $a_1 >0$. Now define new coordinates by $x_i = x_i'$ for $i \geq 2$, and $$x_1 = x_1' +\sum_{\substack{i=2 \\ a_i <0}}^{N-1} \frac{a_i}{a_1}x_i'.$$  If $y \in K_0$, then $ y = \lambda_1 x'_1 + \sum_{i=2}^{N-1} \lambda_i x'_i$ for $\lambda_i \geq a_i$. Substituting to express $y$ with respect to the basis $\{x_i\}$, we obtain
$$y = \lambda_1x_1 + \sum_{\substack{i=2 \\ a_i<0}}^{N-1} (- \frac{\lambda_1}{a_1} a_i +\lambda_i)x_i + \sum_{\substack{i =2 \\ a_i \geq 0}}^{N-1} \lambda_i x_i.$$  All coordinates of vectors in $K_0$ are positive with respect to the new basis, since $- \frac{\lambda_1}{a_1} a_i +\lambda_i >- \frac{\lambda_1}{a_1} a_i +a_i >0 $ for $a_i <0$. \par
The second assertion of the lemma is immediate from \ref{L1}. \end{proof}

\begin{Corollary}\label{edgefact} Given an irreducible effective divisor $D$ on a smooth projective variety $X$ and an irreducible covering family of curves $C$ so that $C \cdot D <0$, the divisor $D$ generates an edge of the effective cone and is therefore extremal. \end{Corollary}
\begin{proof}If we let $H$ denote the class group of $X$ modulo numerical equivalence and identify $\NS(X)= H \otimes \R$ with $\R^N$ for suitable $N$, then $ \Eff(X)= \C(S)$ where $S \subset \NS(X)$ is the set of irreducible effective divisor classes. \par  For any irreducible divisor $D'$, we may choose an irreducible curve $C'$ in the family which does not lie in the intersection of the $D'$ and $D$. Then $C' \cdot D'\geq 0$. This shows that (a) from \ref{L2} is satisfied. It is a well-known fact that the pseudoeffective cone has nonempty, bounded intersection with an appropriately chosen affine hyperplane, so (b) is also satisfied. \end{proof}

%------More Examples ------------------------------------------
\section{Rigid examples and non-extremal examples}\label{rex}

We say that a divisor $D$ is {\bf rigid} if $h^0(kD)= \dim H^0(\mathcal{O}_X(kD))=1$ for all $k\geq 1$. Before commencing with examples, we record the following 

\begin{Lemma}\label{rigiditycrit} Suppose that $X$ is a smooth projective variety, and $D \subset X$ is an irreducible effective divisor with a covering family $\mathcal{F}$ of irreducible curves so that $C \cdot D <0$ for $C \in \mathcal{F}$. Then $D$ is rigid. \end{Lemma}
\begin{proof} See \cite[4.1]{CC} \end{proof}

\begin{Theorem} Given $\z=(a_1, 1,\dots,1,-1,\dots,-1)$, with $a_1>1$,  the divisor $\Lambda_\z$ on $\M_{0,N+2}$ is rigid. Here, $N=a_1 + 2m+1$ and $m$ equals the number of positive $1$'s appearing in the entries of $\z$.  \end{Theorem}
\begin{proof} We proceed by induction on $m$. If $m=1$, then $\Lambda_\z$ is one of the counterexamples provided in \ref{counterex}, and hence there exists an irreducible covering family of curves for the divisor with negative intersection pairing; \ref{rigiditycrit} implies that $\Lambda_\z$ is rigid. \par

Now suppose that  $m>1$. If $ \bf b:= (a_1+1,1, \dots,1,1,-1,-1,\dots,-1)$, by \ref{restricting} we have that $\Lambda_\z \cap \delta_{\{1,2\}} \simeq \Lambda_\zb$ when $\delta_{\{1,2\}}$ is naturally identified with $\M_{0,N+1}$. By induction $\Lambda_\zb$ is rigid. We have an exact sequence

$$0 \lra \mathcal O_{\M_{0,N+2}}(k\Lambda_\z - \delta_{\{1,2\}})\lra \mathcal O_{\M_{0,N+2}}(k \Lambda_\z) \lra \mathcal O_{ \delta_{\{1,2\}} }(k\Lambda_\zb) \lra 0,$$ which gives a long exact sequence in cohomology

$$0 \lra H^0(\mathcal O_{\M_{0,N+2}}(k\Lambda_\z - \delta_{\{1,2\}}))\lra H^0(\mathcal O_{\M_{0,N+2}}(k \Lambda_\z)) \lra H^0( \mathcal O_{ \delta_{\{1,2\}}}(k\Lambda_\zb)) \lra \dots $$ 

Since $\dim H^0( \mathcal O_{ \delta_{\{1,2\}} }(k\Lambda_\zb)) = h^0(k\Lambda_\zb)=1$, it will suffice to show $k \Lambda_\z-\delta_{ \{0,1\} }$ is not effective so that $H^0(\mathcal O_{\M_{0,N+2}}( k \Lambda_\z-\delta_{ \{0,1\} })) =0$. To do this, we exhibit a family of irreducible curves so that \begin{enumerate}
\item[i.] For $C$ in the family, $C \cdot (k \Lambda_\z - \delta_{ \{1,2\} }) = -1$ 
\item[ii.] For a general point of $\M_{0,N+2}$, some curve in the family passes through the point. \end{enumerate} Given the above, the divisor $k \Lambda_\z - \delta_{\{1,2\}}$ cannot be effective: a codimension one subvariety with class $k\Lambda_\z - \delta_{ \{1,2\}}$ must contain each curve $C$ in the family, an absurdity since the curves cover an open subset of $\M_{0,N+2}$. \par

We now construct a family of curves satisfying (i) and (ii). Using formulas for classes with respect to the Kapranov basis in index $1$, given in \ref{classeqs}, we have that
$$\Lambda_\z \sim AH - (A-1) E_2 - (A-1)E_{N}- E_{J^0} - E_{J^1}+...$$ 
where $J^i := \{ 2,\dots, N\} \sm \{2,N,N+1+i\}$ for $i=0$ or $i=1$, and $A := a_1 + 2m-1 $. Other terms contribute to the class, but these are linearly independent and irrelevant.\par

The significance of the terms $E_{J^i}$ is that under $\psi_1$, $E_{J^i}$ is mapped to a codimension $2$ span not containing $p_2$ and not containing $p_{N}$ (these points correspond to $E_2$ and $E_{N}$ under the index $1$ Kapranov map). Given a point $y \in \Pj^{N-1}$ not lying on the lined spanned by $p_2$ and $p_N$, consider the two-plane $T_y=\langle y, p_2, p_N\rangle$. If $y \notin \psi_1(E_{J^i})$, $i=1,2$, there exist a unique $x_i \in \psi_1(E_{J^i}) \cap T_y$.

\begin{Claim}\label{7.2} For a general point $y \in \Pj^{N-1}$, the points $x_i,p_j,$ and $y$ as constructed above lie in general position in $T_y \simeq \Pj^2$. \end{Claim}
 Given this, for general $y$ we have a unique irredicuble conic $C_y \subset T_y$ passing through all five points, and $C_y$ has class $2h + e_2+e_{N} + e_{J^1} + e_{J^0}$ with respect to the dual of the index $1$ Kapranov basis. Pairing with $\Lambda_\z$, we see that $$C_y \cdot \Lambda_\z = 2A - 2(A-1) -2 =0.$$ Since $C_y \cdot \delta_{\{1,2\}}=1$, it follows that $$C_y \cdot (k \Lambda_\z - \delta_{ \{1,2\}} )= -1.$$ Since $C_y$ can be defined for a general point $y$ in $\Pj^{N-1}$, taking the proper transforms of $C_y$ under the $\psi_1$ gives a family of curves in $\M_{0,N}$ satisfying (i) and (ii) above. \par

To prove the claim \ref{7.2}, we first show that $x_0, x_1,$ and $p_i$ are non-collinear for general $y$ and $i \in \{2,N\}$. Consider projection from $p_2$ to $\Pj^{N-2}$. Let $L$ be the image of $T_y$, and $M_i$ be the image of the linear span $\langle p_j \rangle_{j \in J^i}$. Each $M_i$ is of codimension $1$, and $L$ is of dimension $1$. Hence $L \cap M_i$ consists of a single point for each $i$. If $x_0, x_1,$ and $p_2$ are collinear, then $L \cap M_0 = L \cap M_1$. Composing our first projection with a second projection from $p_N$ to obtain a map $\pi \colon \Pj^{N-1} \ra \Pj^{N-3}$, we see that $M_0 \cap L = M_1 \cap L$ occurs if and only if $\pi(y)$ lies in the image of $\langle p_j \rangle_{j \in J^0 \cap J^1}$ under $\pi$, which is of codimension $1$ in $\Pj^{N-3}$.

Note that $p_2, p_N,y$ and $p_2, p_N, x_i$ are non-collinear for general $y$. So, to conclude that the points are in general position, it suffices to verify that $y,x_i,p_j$ are non-collinear for each $i \in \{2,N\}$ and $j \in \{0,1\}$;  By symmetry, we may check only for $i=2$, $j=0$. Note that $x_0$, $y$, and $p_2$ are collinear if and only if the image of $y$ under projection from $p_2$ lies in the image of $\psi_1(E_{J^0})$ under projection. Since $2 \notin J^0$, this image is of codimension $1$ and a general point $y$ is not contained.  \end{proof}

\begin{Remark} While rigidity is not known to imply extremality on $\M_{0,n}$, we are unaware of any examples of rigid, non-extremal divisors on the space in question. Rigidity of a divisor class $D$ implies that $D$ cannot be written as a non-negative linear combination of effective divisors with rational coefficients; for extremality, we must have that $D$ cannot be written as a non-negative linear combination of pseudo-effective divisors. \end{Remark}

We now turn our attention to a class of Chen-Coskun divisors that can be written as linear combinations of other effective divisors, and so are non-rigid and non-extremal. Consider an $n$-tuple of nonzero integers $\z=(a_1, \dots ,a_n)$ with $\sum_i a_i=0$; assume that $a_1 >0$ and $a_n<0$. Define $\tilde{\z} = (a_1 +1,a_2,\dots,a_{n-1}, a_n-1)$. $\Lambda_\z$ and $\Lambda_{\tilde {\bf a}}$ both defined Chen-Coskun divisors on $\M_{0,n+2}$. We compare their classes with respect to the Kapranov basis in index $1$. From class formulas, we see that 

$$\Lambda_{\tilde {\bf a}} = (\sum_{i=2}^n |a_i|) H + ....,$$

and 
$$\Lambda_\z =\big( (\sum_{i=2}^n |a_i|)-1 \big)H + ....$$

Furthermore, we claim that the coefficient of $E_J$ in the class of $\Lambda_\z$ and $\Lambda_{\tilde {\bf a}}$ will be the same whenever $n \in J$: given an $n$-tuple ${\bf b}=(b_i)$ satisfying the appropriate conditions, the coefficient of $E_J$ in the class of $\Lambda_\zb$ is a function of $b_i$ for $i \notin J$ and $i \neq 1$. Since $a_i$ and $\tilde a_i$ agree for $i \neq 1$ and $i \neq n$, the claim follows. \par

With these preliminary observations, we can conclude that if $\Lambda_\z = dH - \sum_I m_I E_I$ and $\Lambda_{\tilde {\bf a}} = eH -\sum_I n_I E_I$, then 

$$\Lambda_{\tilde {\bf a}} - \Lambda_\z = H - \sum_{n \notin I} (n_I -m_I) E_I.$$

Applying \ref{restricting}, we see that for $n,n+1, n+2 \notin J$:

\begin{equation}\label{X1}n_I -m_I = |a_n -1| - |a_n| = 1, \end{equation} and for $n \notin J$, $|\{ n+1, n+2 \cap J\} | =1$, we have that 
$$ n_I - m_I =$$
\begin{equation}\label{X2} \min\bigg \{ \sum_{ \substack{i \notin J\cup \{n\} \\ a_i \leq 0}} \!\! |a_i| + |a_n|+1, \sum_{ \substack{i \notin J \\ a_i \geq 0}}|a_i|\bigg\} - \min \bigg\{ \sum_{ \substack{i \notin J \\ a_i \leq 0}}|a_i|, \sum_{ \substack{i \notin J \\ a_i \geq 0}}|a_i|\bigg\}.\end{equation} If $|a_n| \geq \sum_{ a_i \geq 0}|a_i|,$ then both minima are equal to the positive sum, and the coefficient of $E_J$ is zero. From this, we easily obtain

\begin{Theorem}\label{7.5} For $\z = (a_1, \dots, a_n)$ with $a_i$'s nonzero, $\sum_i a_i=0$, $a_1 >0$, and $a_n<0$, define $\tilde \z = (a_1+1, a_2, \dots, a_{n-1},a_n-1)$. If $$|a_n| \geq \sum_{\substack{a_i \geq 0 \\ i \geq 2}} |a_i|,$$ then $\Lambda_{\tilde {\bf a}} = \Lambda_\z + D$, where $D$ is an effective sum of boundary divisor classes. In particular, $\Lambda_{\tilde {\bf a}}$ is not extremal. \end{Theorem} 

\begin{proof} By the above discussion, $$\Lambda_{\tilde {\bf a}}= \Lambda_\z + H - \sum_{\substack{n \notin J \\ n+1,n+2 \notin J}} E_J.$$ However,

$$ H - \sum_{\substack{n \notin J \\ n+1,n+2 \notin J}} E_J = \delta_{n+1,n+2} + \sum_{\substack{n+1,n+2 \in J \\ n \in J }} E_J$$ which is effective.
\end{proof}

\begin{Example}\label{7.6} If $k \geq dm$ for $k,d,m$ positive integers with $\gcd(k,m)=1$, let $\z(k,d,m)$ be the $2d+2$-tuple $(k, m,m,m, \dots, -m,-m,-m,\dots, -k)$. Then $\Lambda_{\z(k,d,m)} \subset \M_{0,2d+4}$ is not extremal. Indeed, this follows immediately from \ref{7.5}, since the hypothesis that $|a_{2d+2}|>1$ and $$|a_{2d+2}|  \geq \sum_{\substack{ i \geq 2 \\ a_i >0 }} |a_i|$$ is satisfied. \end{Example}
\begin{Remark} The hypothesis of \ref{7.5} are not necessary for non-extremality. The  divisors $L_k$ from \ref{LM} give another class of non-extremal Chen--Coskun divisors, but these do not satisfy the hypotheses of \ref{7.5}. A proof of this is roughly as follows: it was observed in \ref{LM} that for appropriate choice of Kapranov basis, there exist indices $i$ and $j$ so that $E_I$ has zero coefficient in the class of $L_k$ whenever $i \in I$ or $j \in I$. Hence these divisors can be realized as pull-backs of non-boundary (and hence non-extremal) divisors from appropriate Losev-Manin spaces \cite{LM}. The argument generalizes to any Chen--Coskun divisor corresponding to an $n$-tuple with only one positive entry.
\end{Remark} 

We now return to the implications of \ref{7.5}. For fixed $d$, all but finitely many divisors $\Lambda_\z$ on $\M_{0,2d+2}$ for $a$ of the form $(k,m,m,m,\dots, -k, -m,-m,-m, \dots)$ are non-extremal. This is in contrast with the results of \cite{CC}, where divisors on $\M_{1,4}$ arising from $4$-tuples of the form $(k,m,-k,-m)$ were shown to be extremal and yielded the result that $\Eff(\M_{1,n})$ is not finitely generated. These particular $n$-tuples could not yield distinct extremal divisors on $\M_{0,6}$, since $\overline{\operatorname{Eff}}(\M_{0,6})$ is generated by the spherical bipyramid divisor together with boundary classes. In particular, the divisors on $\M_{0,6}$ corresponding to $(k,m,-k,-m)$ are extremal if and only if $k=m=1$.  However, a natural question is whether many extremal rays might arise from ``analogous" $n$-tuples with $n$ sufficiently large. The above discussion rules out certain generalizations. \par 

Moreover, \ref{7.5} provides an obstruction to the construction of ``large families" of extremal Chen--Coskun divisors on $\M_{0,n}$ for $n$ fixed. More precisely, obvious schemes for constructing infinite families of Chen--Coskun divisors can provide only finitely many extremal examples.  For instance, fixing all but two indices of a given $n$-tuple and varying these can yield an extremal divisor for only finitely many choices, since after some point one of the variable entries will become large enough in absolute value so that \ref{7.5} guarantees the divisor is non-extremal. However, more innovative approaches to varying $n$-tuples coupled with finer analysis of combinatorial constraints might yield interesting results.

\end{document}